\documentclass{amsart}

\usepackage{graphicx}
\usepackage[all]{xy}
\usepackage{amsmath}
\usepackage{hyperref}
\usepackage{amsfonts,graphics,amsthm,amsfonts,amscd,latexsym}
\usepackage{epsfig}
\usepackage{flafter}
\usepackage{mathtools}
\usepackage{comment}
\usepackage{stmaryrd}
\usepackage{tabularx}
\usepackage{pst-node}

\usepackage{tikz-cd}
\hypersetup{
    colorlinks=true,    
    linkcolor=blue,          
    citecolor=blue,      
    filecolor=blue,      
    urlcolor=blue           
}
\usepackage{comment}

\usepackage{placeins}

\usepackage{pgfplots}
\pgfplotsset{compat=1.15}
\usepackage{mathrsfs}
\usepackage{tikz}
\usetikzlibrary{graphs,positioning,arrows,shapes.misc,decorations.pathmorphing}

\tikzset{
    >=stealth,
    every picture/.style={thick},
    graphs/every graph/.style={empty nodes},
}

\tikzstyle{vertex}=[
    draw,
    circle,
    fill=black,
    inner sep=1pt,
    minimum width=5pt,
]
\usepackage[position=top]{subfig}
\usepackage{amssymb}
\usepackage{color}

\calclayout

\usetikzlibrary{decorations.pathmorphing}
\tikzstyle{printersafe}=[decoration={snake,amplitude=0pt}]

\newcommand{\Aut}{\operatorname{Aut}}

\newcommand{\Pic}{\operatorname{Pic}}

\newcommand{\Cl}{\operatorname{Cl}}

\newcommand{\Bi}{\operatorname{Bi}}

\newcommand{\GL}{\operatorname{GL}}

\newcommand{\NS}{\operatorname{NS}}

\newcommand{\Vol}{\operatorname{Vol}}
\newcommand{\Herm}{\operatorname{Herm}}
\newcommand{\MW}{\operatorname{MW}}

\newcommand{\PGL}{\operatorname{PGL}}

\newcommand{\free}{\mathrm{free}}

\usepackage{tikz}
\usetikzlibrary{matrix,arrows,decorations.pathmorphing}
\usetikzlibrary{arrows}

\definecolor{uuuuuu}{rgb}{0.26666666666666666,0.26666666666666666,0.26666666666666666}

  \newtheorem{introthm}{Theorem}

  \newtheorem{theorem}{Theorem}[section]
  \newtheorem{lemma}[theorem]{Lemma}
  \newtheorem{proposition}[theorem]{Proposition}
  \newtheorem{corollary}[theorem]{Corollary}

\theoremstyle{definition}

  \newtheorem{definition}[theorem]{Definition}

  \newtheorem{construction}[theorem]{Construction}

\theoremstyle{remark}
\newtheorem{remark}[theorem]{Remark}

\newtheorem{example}[theorem]{Example}

\numberwithin{equation}{section}



\begin{document}
\title[BIANCHI GROUPS AND AUTOMORPHISMS OF RANK-FOUR K3 SURFACES]{BIANCHI GROUPS AND AUTOMORPHISMS OF RANK-FOUR K3 SURFACES}
\author[K.~Hashimoto]{Kenji Hashimoto}
\address{
}
\email{kenji.hashimoto.math@gmail.com}

\author[T.~Oda]{Tomoki Oda}
\address{UCLA Mathematics Department, Box 951555, Los Angeles, CA 90095-1555, USA
}
\email{tomokioda0723@math.ucla.edu}
\thanks{The second author was partially supported by NSF research grant DMS-2443425.}

\begin{abstract}
We relate the arithmetic of Bianchi groups to automorphism groups of
Picard-rank-four $K3$ surfaces. Let $K$ be an imaginary quadratic field
with ring of integers $\mathcal O_K$, and let $S_K=\operatorname{Herm}_2(\mathcal O_K)$
be the rank-four lattice of $2\times2$ Hermitian matrices over
$\mathcal O_K$, equipped with the quadratic form $2\det$. For an odd
integer $N\geq1$, we consider a very general $S_K(2N)$-polarized $K3$
surface $X_{K,2N}$.

We prove that its automorphism group is commensurable with a level-$2N$
congruence subgroup of the Bianchi group. Furthermore, we also obtain exact
realizations of congruence subgroups as full automorphism groups. Namely,
if $K=\mathbb Q(i)$ or $K=\mathbb Q(\sqrt{-p})$, where $p$ is prime,
then
\[
    \operatorname{Aut}(X_{K,2})
    \cong
    P\Gamma_K(2).
\]
Thus, for every prime $p$, the projective principal congruence subgroup
of level $2$ over $\mathcal O_{\mathbb Q(\sqrt{-p})}$ occurs as the full
automorphism group of a Picard-rank-four $K3$ surface. At higher levels,
the full automorphism group may be either the projective principal
congruence subgroup or the strictly larger projective level subgroup
$\operatorname{Bi}_K(2N)$, depending on the arithmetic of the primes
dividing the level.

We further explain these arithmetic groups geometrically. The surfaces
$X_{K,2}$ arise as deformations of the Kummer surfaces
$\operatorname{Km}(E_K\times E_K)$, yielding explicit double-cover models
and genus-one fibrations. For $K=\mathbb Q(\sqrt{-2})$ and
$K=\mathbb Q(\sqrt{-7})$, the automorphism group is generated by
Mordell--Weil translations associated with genus-one fibrations coming
from cusps, together with the covering involution. For $K=\mathbb Q(i)$
and $K=\mathbb Q(\sqrt{-3})$, we construct complete-intersection models
in products of projective spaces and show that their automorphism groups
are generated by deck involutions.
\end{abstract}

\maketitle

\setcounter{tocdepth}{1}
\tableofcontents

\section{Introduction}

A fundamental question in algebraic geometry, broadly posed by Mazur~\cite[Section~7]{Mazur93}, asks to what extent the automorphism groups of projective varieties are governed by arithmetic groups. However, Totaro demonstrated that this expectation fails, even for \(K3\) surfaces of Picard rank four: there exist Picard-rank-four \(K3\) surfaces whose automorphism groups are not commensurable with any arithmetic group~\cite[Theorem~6.1 and Example~6.3]{Tot09}. Despite this general failure, there are also positive realization results. For example, the first author and Lee~\cite{HL25} have shown how to realize congruence subgroups of the modular group as the automorphism group of certain \(K3\) surfaces. These explicit constructions prove that while arithmeticity is not guaranteed, highly structured arithmetic groups can still be captured by the geometry of specific \(K3\) surfaces.

The purpose of this paper is to push this positive realization program into
the realm of three-dimensional hyperbolic geometry. The example constructed
by the first author and Lee has Picard rank three, so that the projectivized
positive cone of the N\'eron--Severi lattice is a model of the hyperbolic plane
\(\mathbb H^2\), and the automorphism group is realized through its action on
this space. In Picard rank four, the corresponding projectivized positive cone
is instead a model of hyperbolic three-space \(\mathbb H^3\), providing a
natural source of discrete groups of hyperbolic isometries. For any \(K3\) surface, the action on the N\'eron--Severi lattice gives a
representation
$\Aut(X)\longrightarrow O(\operatorname{NS}(X)).$ 
In the cases considered in this paper, this representation is injective by
Lemma~\ref{lem-twoelementary}. Hence we can identify $\Aut(X)$ with subgroup of $O(\NS(X))$. We construct explicit examples of
Picard-rank-four \(K3\) surfaces whose automorphism groups are not merely
abstract discrete groups of hyperbolic isometries, but are described precisely
as congruence subgroups of Bianchi groups; see
Theorem~\ref{thm:intro-main}.
\medskip

Let \(K\) be an imaginary quadratic field with ring of integers
\(\mathcal O_K\). We consider the lattice
\[
    S_K := \operatorname{Herm}_2(\mathcal O_K)=  \left\{ A=\begin{pmatrix} a & b\\ \bar b & c \end{pmatrix} : a,c\in\mathbb Z,\ b\in\mathcal O_K \right\}   ,
    \qquad
    q(A):=2\det(A).
\]
This lattice has signature \((1,3)\), and the Bianchi group $\operatorname{Bi}_K:=\operatorname{PGL}_2(\mathcal O_K)$
acts on \(S_K\) by Hermitian conjugation; see Subsection~\ref{subsecBianchi} for this identification. For every odd integer \(N\geq 1\),
we construct very general \(K3\) surfaces \(X_{K,2N}\) satisfying $\operatorname{NS}(X_{K,2N})\cong S_K(2N).$ 

The very generality assumption, together with
Lemma~\ref{lem:very-general-hodge-isometries} and the global Torelli theorem,
gives a concrete criterion for which Bianchi isometries are realized by
automorphisms of \(X_{K,2N}\). Namely, Proposition~\ref{prop:level-containment}
shows that an element of \(\operatorname{Bi}_K\) can be induced by an
automorphism only if its action on the discriminant group
\(A_{S_K(2N)}\) is multiplication by a single global sign \(\pm 1\). This
condition is arithmetic and can be expressed in terms of congruence subgroups.

For \(m\geq 1\), we define the projective level-\(m\) congruence subgroup
\[
    \Bi_K(m)
    :=
    \ker\!\left(
        \operatorname{PGL}_2(\mathcal O_K)
        \longrightarrow
        \operatorname{PGL}_2(\mathcal O_K/m\mathcal O_K)
    \right).
\]
 We also consider the
finer projective principal congruence subgroup
\[
    P\Gamma_K(m)
    :=
    \left\{
        [M]\in \operatorname{PGL}_2(\mathcal O_K)
        \,\middle|\,
        \begin{array}{c}
        \text{some representative }M\in\operatorname{GL}_2(\mathcal O_K)\\
        \text{satisfies }M\equiv I_2\pmod{m\mathcal O_K}
        \end{array}
    \right\}.
\]
There is always an inclusion
$P\Gamma_K(m)\subseteq \Bi_K(m).$
In particular, when \(m=2\) and
\(\mathcal O_K^\times=\{\pm1\}\), the two groups agree if \(2\) is split
or inert, whereas they may differ if \(2\) is ramified; see
Proposition~\ref{cor:Bi-vs-PGamma-2}.
\medskip 

Our first result compares the automorphism group with the arithmetic
congruence subgroups introduced above. It shows that the part of the
automorphism group lying in the Bianchi group is contained in
\(\Bi_K(2N)\) and, under an additional unramifiedness assumption at \(2\),
has explicitly bounded index. The index bound measures the discrepancy between the congruence subgroup
determined by the level structure and the subgroup satisfying the
global-sign condition on the discriminant group from
Proposition~\ref{prop:level-containment}. This is a uniform estimate that depends only on level $2N$. 

\begin{introthm}[Corollaries~\ref {Bianchicontainment}, \ref{cor:index-bound}]
\label{thm:intro-main}
Let \(K\) be an imaginary quadratic field with fundamental discriminant
\(-D\), and let \(N\geq 1\) be an odd integer such that every prime divisor
of \(N\) is unramified in \(K\). Let $X_{K,2N}$ be a very general \(S_K(2N)\)-polarized \(K3\) surface. Then the following
statements hold:
\begin{enumerate}
    \item $\Bi_K\cap\Aut(X_{K,2N})
        \subseteq
        \Bi_K(2N).$
        \item If \(2\) is also unramified in \(K\), then
    \[
        \left[
            \Bi_K(2N):
            \Bi_K\cap\Aut(X_{K,2N})
        \right]
        \leq
        2^{\omega(N)},
    \]
    where \(\omega\) denotes the number of distinct
    prime divisors of \(N\).
\end{enumerate}
\end{introthm}
The intersections in parts~\((1)\) and~\((2)\) are necessary because,
with our definitions, \(\Aut(X_{K,2N})\) is not a priori realized as a
subgroup of the Bianchi group; see Proposition~\ref{Lem:extendedBianchi}.
Thus, the theorem first identifies the portion of the automorphism group
lying in \(\Bi_K\) and then compares it with the projective level subgroup
\(\Bi_K(2N)\).

\medskip

Our second result determines the full automorphism group at level \(2\)
for the fields \(K=\mathbb Q(i)\) and
\(K=\mathbb Q(\sqrt{-p})\), where \(p\) is prime.

\begin{introthm}
\label{realization}
Let $K=\mathbb Q(i)$ \text{or} $
    K=\mathbb Q(\sqrt{-p}),$
where \(p\) is a prime. Then $\Aut(X_{K,2})
    \cong
    P\Gamma_K(2).$
\end{introthm}

The proof depends on the arithmetic behavior of the prime \(2\) in \(K\).
When \(2\) is unramified, the result follows from the local
discriminant-form computation and the index estimate in
Theorem~\ref{thm:intro-main}. When \(2\) is ramified, a separate level-two
analysis is required; this is carried out in
Lemmas~\ref{lem:pgamma-sign}--\ref{lem:extended-coset}. One must also
account for the fact that the natural lattice representation places the
automorphism group inside the extended Bianchi group a priori.
Proposition~\ref{Lem:extendedBianchi} and
Lemma~\ref{lem:extended-coset} show that the additional extended-Bianchi
cosets do not produce further automorphisms in the cases under
consideration.

As a consequence of Theorem~\ref{realization}, for every prime \(p\), the
group $P\Gamma_{\mathbb Q(\sqrt{-p})}(2)$
occurs as the full automorphism group of a Picard-rank-four \(K3\)
surface.

\medskip

We next give exact realizations at higher level. Here two distinct
phenomena occur. In the first family, the full automorphism group is the
projective principal congruence subgroup. In the second, it is the
strictly larger projective level subgroup \(\Bi_K(2N)\).

\begin{introthm}
[Theorems~\ref{prop:split-prime-power-principal-congruence}, ~\ref{thm:prime-condition-Aut-Bi}]
\label{thm:intro-higher-level}
The following statements hold under the normalized Bianchi action:
\begin{enumerate}
    \item Let \(p\) be a prime satisfying  $p\equiv 7\pmod 8,$
    and let \(K=\mathbb Q(\sqrt{-p})\). Suppose that $N=q^e,$ where \(q\equiv 3\pmod 4\) is an odd prime that splits in \(K\).
    Then
    \[
        \Aut(X_{K,2N})
        \cong
        P\Gamma_K(2N).
    \]

    \item Let \(K=\mathbb Q(\sqrt{-p})\), where \(p>3\) is a prime
    satisfying $p\equiv 3\pmod 4.$ Suppose that \(N>1\) is odd and that every prime divisor
    \(\ell\mid N\) is inert in \(K\) and satisfies $\ell\equiv 1\pmod 4.$
    Then
    \[
        P\Gamma_K(2N)
        \subsetneq
        \Bi_K(2N)
  \cong
        \Aut(X_{K,2N}).
    \]
\end{enumerate}
\end{introthm}
The two parts of Theorem~\ref{thm:intro-higher-level} exhibit a genuine
higher-level dichotomy. In part~\((1)\), where the prime dividing \(N\)
splits in \(K\), the full automorphism group is the projective principal
congruence subgroup. In part~\((2)\), where every prime dividing \(N\)
is inert in \(K\), the additional projective level elements survive the
global-sign condition of Proposition~\ref{prop:level-containment}, and
the full automorphism group is the strictly larger group
\(\Bi_K(2N)\). Thus, part~\((2)\) shows that
\(\Bi_K(2N)\), rather than \(P\Gamma_K(2N)\), is the natural ambient
congruence subgroup in the general higher-level theory.

\medskip 
The first part of the paper identifies the automorphism group arithmetically.
The second part explains how this arithmetic group is realized on the surface
itself. After projectivizing the
positive cone, automorphisms of a \(K3\) surface act by isometries of hyperbolic
space, and these isometries are elliptic, parabolic, or loxodromic
\cite[Exercise 4.7-18]{Rac19}. In the present paper, the parabolic elements and involutions play central roles. Proposition~\ref{prop:cusp_translations_XK2N} shows that the parabolic
elements are precisely those preserving the genus-one fibrations, and that
their unipotent actions are given by the corresponding Mordell--Weil translation
groups. The distinguished involution is the deck involution of the double-cover
model. It is also tied to the genus-one fibrations; see
Lemma~\ref{lem:covering-involution-Hermitian-action}, Subsection~\ref{subsec:MW}, and Subsection~\ref{subsec:rank-four-complete-intersections}, especially Lemmas~\ref{lem:degree-two-product-fibrations}--\ref{lem:deck-involution-edge-isometry}.
\medskip 

At level \(2\), the surfaces \(X_{K,2}\) arise from Kummer surfaces. Let \(E_K \cong \mathbb{C}/\mathcal{O}_K\) be the elliptic curve with complex multiplication by \(\mathcal{O}_K\), and set \(A_K := E_K \times E_K\). The Kummer surface \(\operatorname{Km}(A_K)\) contains a natural non-exceptional lattice isomorphic to \(S_K(2)\). Deforming the Kummer surface while preserving this lattice produces the \(S_K(2)\)-polarized surface \(X_{K,2}\). This construction is useful because it also produces explicit double-cover models and genus-one fibrations.

\begin{introthm}[Proposition~\ref{prop:main_existence} and Corollary~\ref{cor:final_generation}]\label{thm:intro_kummer_deformation}
Let \(K=\mathbb{Q}(\sqrt{-d})\), and let \(E_K \cong \mathbb{C}/\mathcal{O}_K\) be the elliptic curve with complex multiplication by \(\mathcal{O}_K\). Set \(A_K := E_K \times E_K\). Then the following hold.
\begin{enumerate}
    \item Choose a double-cover model $E_K = \left\{ s^2 = B(x,y) \right\} \subset \mathbb{P}(1,1,2),$
    where \(B(x,y)\) is a binary quartic form. Then \(X_{K,2}\) can be realized as a double cover of \(\mathbb{P}^1\times \mathbb{P}^1\) of the form
    \[
        X_{K,2} = \left\{ S^2 = B(x,y)B(z,w)+H_{4,4}(x,y;z,w) \right\} \longrightarrow \mathbb{P}^1 \times \mathbb{P}^1,
    \]
    where \(H_{4,4}\) is a bihomogeneous form of bidegree \((4,4)\) satisfying the corresponding Noether--Lefschetz condition. There is a covering involution $\iota \colon S \longmapsto -S.$
\item For $K=\mathbb{Q}(\sqrt{-2})$ \text{and} $K=\mathbb{Q}(\sqrt{-7}),$
the group \(\operatorname{Aut}(X_{K,2})\) is generated by Mordell--Weil translation groups associated with genus-one fibrations, together with the covering involution
$\iota \colon S \longmapsto -S.$
\end{enumerate}
\end{introthm}
This involution is particularly significant. By 
Lemma~\ref{lem:level-two-torsion}, although there are infinitely many involutions, the complete-intersection
models constructed below exhibit finite collections of geometrically natural
deck involutions that generate the full automorphism group in the Gaussian
and Eisenstein cases.

\begin{introthm}[cf.~Theorem~\ref{lem:period-maps-dominant} and ~\ref{thm:deck-generates-automorphism}]
Let $K=\mathbb{Q}(i)$ \text{or} $K=\mathbb{Q}(\sqrt{-3})$.
Then \(X_{K,2}\) admits a complete-intersection model in a product of projective spaces, and \(\operatorname{Aut}(X_{K,2})\) is generated by deck involutions.
\end{introthm}

In the Gaussian case, these deck involutions act on hyperbolic three-space as edge half-turns in the regular ideal octahedral tiling. In the Eisenstein case, they act as edge half-turns in the regular ideal tetrahedral tiling (see Proposition~\ref{prop:natural-involutions-type-I-II}). Hence, in these exceptional unit cases, the Bianchi group is visible not only through the lattice \(S_K(2)\), but also through  projective symmetries of the corresponding \(K3\) surface.

\subsection*{Acknowledgements}
The second author is grateful to Burt Totaro, Joaquín Moraga and Zachary Baugher for helpful discussions related to this project, and to Professor Brandhorst, Professor Dolgachev, Professor Oguiso and Professor Roulleau for valuable email correspondence.
\section{Preliminaries}
\label{sec:preliminaries}

Throughout this article, we work over $\mathbb{C}$ . All $K3$ surfaces are assumed to be projective. The $K3$ lattice is defined by $\Lambda_{K3} := U^{\oplus 3} \oplus E_8(-1)^{\oplus 2}$.
Since \(H^1(X,\mathcal O_X)=0\), we identify \(\operatorname{Pic}(X)\)
with the Néron--Severi lattice \(\operatorname{NS}(X)\) via the first Chern class. We shall use the following theorem of Morrison, which guarantees the existence
of complex projective \(K3\) surfaces with prescribed low-rank Picard lattice.

\begin{theorem}[{\cite[cf.~Corollary~2.9]{Mor84}}]
\label{thm:existence-picard-lattice}
Let \(L\) be an even lattice of signature \((1,\rho-1)\) with \(\rho\le 10\). Then there exists a projective \(K3\) surface \(X\) such that $\Pic(X)\cong L.$
\end{theorem}

We recall the basic language of discriminant forms and primitive embeddings.

\begin{definition}
  Let $L$ be an integral, nondegenerate, even lattice. Its \emph{dual lattice} is
  \[
    L^\vee := \{x \in L \otimes \mathbb{Q} \mid (x,y) \in \mathbb{Z} \text{ for all } y \in L\}.
  \]
  Because $L$ is integral and even, the \emph{discriminant group} $A_L := L^\vee/L$ carries a natural finite quadratic form:
  \[
    q_L: A_L \longrightarrow \mathbb{Q}/2\mathbb{Z}, \qquad q_L(x+L) = (x,x) \pmod{2\mathbb{Z}}
  \]
\end{definition}

Every isometry $g \in O(L)$ canonically induces an isometry $\bar{g} \in O(A_L, q_L)$. 

\begin{proposition}[{\cite[Proposition~1.6.1]{Nik80}}]
\label{prop:nikulin-primitive-embedding}
Let $\Lambda$ be an even unimodular lattice. If $S \subset \Lambda$ is a nondegenerate primitive  sublattice and $T = S^\perp$ is its orthogonal complement, then $\Lambda$ induces a canonical anti-isometry
\[
  (A_S, q_S) \cong (A_T, -q_T)
\]
between their discriminant forms.
\end{proposition}
Let $S$ be a Picard lattice and let $T := S^\perp_{\Lambda_{K3}}$ be its orthogonal complement. The corresponding period domain is:
\[
  \Omega_S := \left\{ [\omega] \in \mathbb{P}(T \otimes \mathbb{C}) \ \middle|\ (\omega,\omega)=0, \quad (\omega,\overline{\omega})>0 \right\}
\]

The next lemma says that, for a very general period in \(\Omega_S\), the
transcendental Hodge structure has the smallest possible group of Hodge
isometries.
\begin{lemma}
\label{lem:very-general-hodge-isometries}
Let \(S\) be a primitive hyperbolic sublattice of \(\Lambda_{K3}\), and put
$T:=S^\perp_{\Lambda_{K3}}$.
For a very general period point \([\omega]\in \Omega_S\), the group of Hodge
isometries of the Hodge structure \(T_\omega\) is $ O_{\mathrm{Hdg}}(T_\omega)=\{\pm \mathrm{id}\}.$
\end{lemma}

\begin{proof}
For \(g\in O(T)\), define
$\operatorname{Fix}(g)
  :=
  \{[\omega]\in \Omega_S \mid g(\mathbb C\omega)=\mathbb C\omega\}.$
This is a closed analytic subset of \(\Omega_S\), since the condition
\(g(\mathbb C\omega)=\mathbb C\omega\) is equivalent to $g\omega\wedge \omega=0$ in
  $\bigwedge^2 T_{\mathbb C}.$

We claim that if \(\operatorname{Fix}(g)=\Omega_S\), then \(g=\pm \mathrm{id}\).
Indeed, in that case \(g\) fixes every period line in \(\Omega_S\). Hence the
induced projective transformation of \(\mathbb P(T_{\mathbb C})\) is the
identity on a nonempty open subset of the period quadric. Since this quadric
spans \(\mathbb P(T_{\mathbb C})\), the transformation is the identity on
\(\mathbb P(T_{\mathbb C})\). Therefore \(g\) acts as a scalar on
\(T_{\mathbb C}\). Since \(g\) is an isometry of the lattice \(T\), this scalar
must be \(\pm 1\). Thus \(g=\pm \mathrm{id}\).

Consequently, for every \(g\in O(T)\setminus\{\pm \mathrm{id}\}\), the locus
\(\operatorname{Fix}(g)\) is a proper closed analytic subset of \(\Omega_S\).
Since \(O(T)\) is countable, the union
$\bigcup_{g\in O(T)\setminus\{\pm \mathrm{id}\}} \operatorname{Fix}(g)$
is a countable union of proper closed analytic subsets. A very general period
point \([\omega]\in \Omega_S\) avoids this union.

Now let \(h\in O_{\mathrm{Hdg}}(T_\omega)\). Since \(h\) is a Hodge isometry,
it preserves the line $H^{2,0}(T_\omega)=\mathbb C\omega.$
Hence \([\omega]\in \operatorname{Fix}(h)\). By the choice of \([\omega]\), this
implies \(h=\pm \mathrm{id}\). Conversely, both \(\mathrm{id}\) and
\(-\mathrm{id}\) are Hodge isometries. Therefore
$O_{\mathrm{Hdg}}(T_\omega)=\{\pm \mathrm{id}\}.$
\end{proof}

We use the following version of the Torelli theorem to compute the automorphism group of the $K3$ surfaces.

\begin{theorem}\cite[Chapter~15 Corollary~2.3]{Huy16}
\label{thm:torelli_mod_weyl}
Let \(X\) be a projective K3 surface and put $S=\NS(X) \text{ and }  T=S^\perp\subset H^2(X,\mathbb Z).$
Let $\Delta(S)=\{\delta\in S\mid \delta^2=-2\},$
and let \(W(S)\subset O^+(S)\) be the Weyl group generated by the reflections for $\Delta(S)$.

Let \(\Gamma_X\subset O^+(S)\) be the subgroup of isometries which are
compatible with the Hodge structure on \(T\), namely
\[
\Gamma_X
=
\left\{
g\in O^+(S)\ \middle|\ 
\begin{array}{c}
\text{there exists }h\in O_{\mathrm{Hdg}}(T)
\text{ such that }g\oplus h\\
\text{extends to an isometry of }H^2(X,\mathbb Z)
\end{array}
\right\}.
\]
Then \(W(S)\triangleleft \Gamma_X\), and there is a natural exact sequence
\[1 \longrightarrow \ker\bigl(\operatorname{Aut}(X)\to O(S)\bigr)
\longrightarrow \operatorname{Aut}(X)
\longrightarrow \Gamma_X/W(S)
\longrightarrow 1.\]
In particular, if the restriction homomorphism $\Aut(X)\longrightarrow O(S)$
is injective, then $\Aut(X)\cong \Gamma_X/W(S).$
\end{theorem}

We use the following injectivity criterion
for the automorphism groups.

\begin{lemma}\label{lem-twoelementary}
Let $X$ be a very general $S$-polarized projective $K3$ surface, and
assume that the discriminant group $A_S$ is not $2$-elementary. Then the kernel of the natural representation
 $\operatorname{Aut}(X)\longrightarrow O(S)$ is trivial. 
\end{lemma}
\begin{proof}
Let $f \in \operatorname{Aut}(X)$ be in the kernel, so $f^*|_S = \operatorname{id}_S$. Then $f^*$ restricts to a Hodge isometry of the transcendental lattice $T_X = S^\perp$. Since $X$ is very general, Lemma~\ref{lem:very-general-hodge-isometries} gives $f^*|_{T_X} = \pm \operatorname{id}_{T_X}$. 

If $f^*|_{T_X} = +\operatorname{id}_{T_X}$, $f^*$ is the identity on the full lattice $H^2(X, \mathbb{Z})$. By the Global Torelli theorem, $f = \operatorname{id}_X$, giving the trivial kernel $\{1\}$.

If $f^*|_{T_X} = -\operatorname{id}_{T_X}$, the split isometry $\operatorname{id}_S \perp (-\operatorname{id}_{T_X})$ extends to the unimodular overlattice $H^2(X, \mathbb{Z})$ if and only if the actions on the discriminant groups match. By Proposition~\ref{prop:nikulin-primitive-embedding}, this requires $\operatorname{id}_{A_S} = -\operatorname{id}_{A_S}$, which means $2x = 0$ for all $x \in A_S$ i.e., $A_S$ is $2$-elementary.
\end{proof}

\section{Bianchi groups and automorphism groups}
\subsection{Bianchi Groups}\label{subsecBianchi}
Throughout this article, let $d>0$ be a squarefree integer, and set $K=\mathbb Q(\sqrt{-d})$. Let $\mathcal O_K$ be the ring of integers of $K$, and let $-D$ be the fundamental discriminant of $K$. We consider the $\mathbb Z$-module
$$
S_K := \Herm_2(\mathcal{O}_K)= \left\{ A=\begin{pmatrix} a & b\\ \bar b & c \end{pmatrix} : a,c\in\mathbb Z,\ b\in\mathcal O_K \right\},
$$
equipped with the quadratic form $q_d(A) := 2\det(A)$. 

This module $S_K$ can be naturally identified with a rank-four lattice equipped with an intersection form. To maintain notation independent of the choice of $d$, we will write $S_K \cong U \perp M_K$, where $U$ is the standard hyperbolic plane and $M_K$ is a rank-$2$ negative-definite lattice defined according to the fundamental discriminant:
\begin{enumerate}
    \item If $-D \equiv 0 \pmod 4$, then $M_K := \langle -2 \rangle \perp \langle -\frac{D}{2} \rangle$.
    \item If $-D \equiv 1 \pmod 4$, then $M_K := \Lambda_D$, where $\Lambda_D$ is the lattice with Gram matrix $\begin{pmatrix} -2 & -1\\ -1 & -\frac{1+D}{2} \end{pmatrix}$.
\end{enumerate}
\begin{definition}
The \emph{Bianchi group} $\Bi_K$ is defined as $\PGL_2(\mathcal{O}_K) = \GL_2(\mathcal{O}_K)/\mathcal{O}_K^\times$. \end{definition}

The Bianchi group acts on \(S_K=\operatorname{Herm}_2(\mathcal{O}_K)\) by
\[
[M]\cdot H := MHM^{\ast},
\qquad
M\in \operatorname{GL}_2(\mathcal{O}_K),
\quad
H\in \operatorname{Herm}_2(\mathcal{O}_K),
\]
where \(M^{\ast}:={}^{t}\overline{M}\).
This action is well-defined on \(\operatorname{PGL}_2(\mathcal{O}_K)\), since multiplying \(M\) by a unit
\(\lambda\in \mathcal{O}_K^{\times}\) multiplies \(MHM^{\ast}\) by
\(\lambda\overline{\lambda}=1\).
Moreover, $\det(MHM^{\ast})
=
\det(M)\overline{\det(M)}\det(H)
=
\det(H),$
because \(\det(M)\in \mathcal{O}_K^{\times}\). Hence this gives a homomorphism \[\operatorname{Bi}_K
\longrightarrow
O(S_K).\]
We choose the positive cone component \(C_K^+\) to be the component containing the positive definite
Hermitian matrices. Since positive definiteness is preserved by $H\longmapsto MHM^{\ast},$
the image of \(\operatorname{Bi}_K\) lies in \(O^+(S_K)\).

\begin{remark}In the literature (e.g., Vinberg \cite{V87}), the \emph{Bianchi group} $\Bi_K$ is often defined by the semidirect product of $\PGL_2(\mathcal{O}_K)$ and the non-trivial Galois automorphism $\sigma$ of $K/\mathbb{Q}$. Geometrically, $\sigma$ acts as an anti-holomorphic isometry. Lemma~\ref{lem:galois-coset-exclusion} shows that this Galois coset does not contribute automorphisms of the very general \(K3\) surfaces considered here.
\end{remark}
Vinberg~\cite{V87} proved that $O^+(S_K)
    =
    \widehat{\operatorname{Bi}}_K
    \rtimes
    \langle \sigma\rangle,$
where $\widehat{\operatorname{Bi}}_K$ denotes the orientation-preserving
extended Bianchi group and $\sigma$ is induced by complex conjugation.  $\widehat{\Bi}_K$ 
 acts by orientation-preserving isometries on $\mathbb{H}^3$
 through the normalized action
\[
    A
    \longmapsto
    \frac{1}{|\det M|}MAM^\ast.
\]
 The normalized action realizes
$\widehat{\operatorname{Bi}}_K$ as a subgroup of $O^+(S_K)$. $\widehat{\operatorname{Bi}}_K$ also contains projective classes of
matrices $M\in\operatorname{GL}_2(K)$ satisfying $M\mathcal O_K^2=\mathfrak a\,\mathcal O_K^2$
for a fractional ideal $\mathfrak a$ whose ideal class has order dividing
two. By \cite[Section~4]{BM13}, there is an exact sequence:
$$1 \longrightarrow \Bi_K\longrightarrow \widehat{\Bi}_K\longrightarrow \Cl(\mathcal{O}_K)[2] \longrightarrow 1.$$
Here $\Cl(\mathcal{O}_K)[2]$  denotes the $2$-torsion subgroup of the ideal class group of $\mathcal{O}_K$. Consequently, using the order estimate of the $2$-torsion group (for example see~\cite[Theorem~6.1]{Cox}), we have the following statement.

\begin{proposition}\label{Lem:extendedBianchi}
Let $K=\mathbb{Q}(\sqrt{-d})$ with $d$ square-free. Then $\Bi_K$ is a normal subgroup of $\widehat{\Bi}_K$, and their index is given by
$[\widehat{\Bi}_K:\Bi_K] = 2^{\omega(-D)-1}$
 where $\omega(-D)$ is the number of distinct primes dividing the discriminant $-D$.
\end{proposition}

\begin{example}\label{exampleextemdedBianchi}
Let $K=\mathbb Q(\sqrt{-p}),$ $p\equiv 1\pmod 4,$
and write $\tau=\sqrt{-p}$. Then $\mathcal O_K=\mathbb Z[\tau],$ $
    \operatorname{disc}(K)=-4p.$ By the preceding proposition, we have $[\widehat{\operatorname{Bi}}_K:
      \operatorname{Bi}_K]=2.$
The nontrivial class is represented by the ramified prime ideal $\mathfrak p_2=(2,1+\tau),$ $
    \mathfrak p_2^2=(2).$ Set
\[
    W_2=
    \begin{pmatrix}
        1+\tau&2\\
        2s&1-\tau
    \end{pmatrix},
    \qquad
    s=\frac{p-1}{4}.
\]
Then $\det(W_2)=2$ and a direct ideal calculation gives  $W_2\mathcal O_K^2
    =
    \mathfrak p_2\mathcal O_K^2.$
Since $\mathfrak p_2$ is nonprincipal, no scalar multiple of $W_2$
belongs to $\operatorname{GL}_2(\mathcal O_K)$. Hence
\[
    [W_2]
    \in
    \widehat{\operatorname{Bi}}_K
    \setminus
    \operatorname{Bi}_K,
\]
and $[W_2]$ represents the unique nontrivial coset of
$\widehat{\operatorname{Bi}}_K/\operatorname{Bi}_K$. Nevertheless, the normalized action
\[
    A\longmapsto
    \frac{1}{2}W_2AW_2^\ast
\]
preserves the integral Hermitian lattice $S_K$. Indeed, the entries of
$W_2$ lie in $\mathfrak p_2$, and therefore
\[
    W_2AW_2^\ast
    \in
    \mathfrak p_2\overline{\mathfrak p}_2
    \operatorname{Herm}_2(\mathcal O_K)
    =
    2\operatorname{Herm}_2(\mathcal O_K).
\]
Thus the factor $1/2$ restores integrality, and the resulting
transformation preserves the determinant quadratic form.
\end{example}

We compute the discriminant group $A_{S_K(2N)}$.

\begin{lemma}
\label{lem:dual_generators_Hermitian_lattice}
Let $K$ be an imaginary quadratic field of discriminant $-D$, $D>0$. For an integer $N \geq 1$, consider the scaled lattice $S_K(2N)$. Let
$$e=\begin{pmatrix}1&0\\0&0\end{pmatrix}, \qquad f=\begin{pmatrix}0&0\\0&1\end{pmatrix}, \qquad P(z)=\begin{pmatrix}0&z\\ \bar z&0\end{pmatrix}.$$
On $S_K(2N)$, the elements $e,f$ span a scaled copy of $U$, and the classes
$$\frac{e}{2N},\qquad\frac{f}{2N}$$
generate the $U(2N)^\vee/U(2N)$-part of the discriminant group. The remaining generators are as follows:
\begin{enumerate}
    \item If $D\equiv 0 \pmod 4$, let $\tau=\frac{\sqrt{-D}}{2}$. The off-diagonal block is generated by $P(1)$ and $P(\tau)$, yielding the corresponding dual generators:
    $$\frac{P(1)}{4N},\qquad \frac{P(\tau)}{DN}.$$

    \item If $D\equiv 3 \pmod 4$, let $\tau=\frac{1+\sqrt{-D}}{2}$. In the basis $P(1),P(\tau)$, the corresponding dual generators are:
    $$\frac{(D+1)P(1)-2P(\tau)}{4DN}, \qquad \frac{-2P(1)+4P(\tau)}{4DN}.$$
\end{enumerate}
\end{lemma}

\begin{proof}
The bilinear form on $S_K(2N)$ is $(X,Y)_{2N}=2N(\det(X+Y)-\det(X)-\det(Y))$.
Hence $(e,e)_{2N}=(f,f)_{2N}=0$ and $(e,f)_{2N}=2N$. Therefore $e,f$ span a scaled hyperbolic plane, and its dual is directly generated by $e/(2N)$ and $f/(2N)$.

For the off-diagonal matrices, $(P(z),P(w))_{2N}=-2N\operatorname{Tr}_{K/\mathbb Q}(z\bar w)$.

If $D\equiv 0\pmod 4$, with $\tau=\sqrt{-D}/2$, the Gram matrix of $P(1),P(\tau)$ scaled by $N$ is
$$\begin{pmatrix} -4N&0\\ 0&-DN \end{pmatrix}.$$
Inverting this matrix immediately yields the dual generators $P(1)/(4N)$ and $P(\tau)/(DN)$.

If $D\equiv 3\pmod 4$, with $\tau=(1+\sqrt{-D})/2$, the Gram matrix of $P(1),P(\tau)$ scaled by $N$ is
$$\Lambda_{D,N}= N \begin{pmatrix} -4&-2\\ -2&-(D+1) \end{pmatrix}\qquad \Lambda_{D,N}^{-1} = \frac{1}{4DN} \begin{pmatrix} -(D+1)&2\\ 2&-4 \end{pmatrix}.$$
Thus the dual basis is generated, up to signs, by
$$\frac{(D+1)P(1)-2P(\tau)}{4DN}, \qquad \frac{-2P(1)+4P(\tau)}{4DN}.$$
This proves the claim.
\end{proof}

\subsection{Automorphism groups of $K3$ surfaces}
\label{subsec:automorphism-groups}

For an integer $m \geq 1$, we define the congruence subgroups of level $m$ as follows:

\begin{definition}
Let $m \geq 1$ be an integer. The \emph{congruence subgroup} of level $m$ is defined as
\[ \operatorname{Bi}_K(m) := \ker\!\left( \operatorname{PGL}_2(\mathcal{O}_K) \longrightarrow \operatorname{PGL}_2(\mathcal{O}_K/m\mathcal{O}_K) \right). \]
We define the \emph{principal congruence subgroup} of level $m$ as
\[
    P\Gamma_K(m)
    :=
    \left\{
        [M]\in \operatorname{PGL}_2(\mathcal O_K)
        \,\middle|\,
        \begin{array}{c}
        \text{some representative }M\in\operatorname{GL}_2(\mathcal O_K)\\
        \text{satisfies }M\equiv I_2\pmod{m\mathcal O_K}
        \end{array}
    \right\}.
\]
\end{definition}

We first compare the two congruence subgroups introduced above.

\begin{proposition}
\label{cor:Bi-vs-PGamma-2}
Let $K$ be an imaginary quadratic field. The index $[\operatorname{Bi}_K(2) : P\Gamma_K(2)]$ and the structure of $\operatorname{Bi}_K(2)$ depend on the factorization of $2$ in $\mathcal{O}_K$ as follows:
\begin{enumerate}
    \item If $2$ is unramified, then $\operatorname{Bi}_K(2) = P\Gamma_K(2)$. 
    \item If $2$ ramifies, then $[\operatorname{Bi}_K(2) : P\Gamma_K(2)] \le 2$.
\end{enumerate}
\end{proposition}

\begin{proof}
Let $R_K = \mathcal{O}_K/2\mathcal{O}_K$. For any $[M] \in \operatorname{Bi}_K(2)$, its reduction modulo $2$ is a scalar matrix $\overline{M} = a I_2$ for some $a \in R_K^\times$. Scaling the representative $M$ by a global unit $u \in \mathcal{O}_K^\times$ multiplies $a$ by $\overline{u}$, yielding a well-defined, injective homomorphism
\[
    \operatorname{Bi}_K(2)/P\Gamma_K(2) \hookrightarrow R_K^\times/\overline{\mathcal{O}_K^\times}.
\]

Suppose first that $2$ is unramified. Then either $2$ splits or $2$ is inert.
If $2$ splits, $R_K \cong \mathbb{F}_2 \times \mathbb{F}_2$, so $R_K^\times = \{1\}$. Thus $a = 1$, and $[M] \in P\Gamma_K(2)$.
If $2$ is inert, $R_K \cong \mathbb{F}_4$, and $a^2 = \overline{\det(M)}$. If $K \neq \mathbb{Q}(\sqrt{-3})$, then $\mathcal{O}_K^\times = \{\pm 1\}$, so $\overline{\det(M)} = 1$. Since $\mathbb{F}_4^\times$ has order $3$, the equality $a^2 = 1$ implies $a = 1$. If $K = \mathbb{Q}(\sqrt{-3})$, the reduction homomorphism $\mathcal{O}_K^\times \to \mathbb{F}_4^\times$ is surjective. We may therefore choose $u \in \mathcal{O}_K^\times$ such that $\overline{u} = a^{-1}$. Then $uM \equiv I_2 \pmod{2\mathcal{O}_K}$, meaning $uM$ represents the same element in $\operatorname{PGL}_2(\mathcal{O}_K)$. Thus $[M] \in P\Gamma_K(2)$. This establishes $\operatorname{Bi}_K(2) = P\Gamma_K(2)$ in the unramified case. 

Finally, suppose that $2$ ramifies. In this case, the local ring structure gives $R_K \cong \mathbb{F}_2[\epsilon]/(\epsilon^2)$, so $|R_K^\times| = 2$. The injective homomorphism therefore bounds the index to $[\operatorname{Bi}_K(2) : P\Gamma_K(2)] \le 2$.
\end{proof}

We now turn to $K3$ surfaces. Throughout this subsection, $N$ is an odd positive integer. Let $X_{K,2N}$ be a very general K3 surface with $\operatorname{NS}(X_{K,2N}) \cong S_K(2N)$.

The guiding principle is that a Bianchi isometry is induced by an automorphism of $X_{K,2N}$ only when its action on the discriminant group matches the action of a Hodge isometry on the transcendental lattice. Since $X_{K,2N}$ is very general, the only Hodge isometries of the transcendental lattice are $\pm\operatorname{id}$. Thus, the $K3$ condition is a \emph{global sign condition} on the discriminant group. 

\begin{proposition}\label{prop:level-containment}
Let $N$ be an odd positive integer, and let $X$ be a very general $K3$ surface with $S := \operatorname{NS}(X) \simeq S_K(2N)$. An isometry $g \in O(S)$ is induced by an automorphism of $X$ if and only if $g$ preserves the ample cone and its induced action on the discriminant group $A_S$ is multiplication by a single sign $\pm 1$.
\end{proposition}

\begin{proof}
If $g = f^\ast|_S$ for $f \in \operatorname{Aut}(X)$, then $g$ trivially preserves the ample cone. Since $X$ is very general, $f^\ast|_T = \pm \operatorname{id}_T$ by Lemma~\ref{lem:very-general-hodge-isometries}. Compatibility via the canonical anti-isometry $A_S \to A_T$ of discriminant forms forces $g$ to act on $A_S$ by the same sign $\pm 1$.

Conversely, if $g$ preserves the ample cone and acts on $A_S$ by a sign $\sigma \in \{\pm 1\}$, setting $h := \sigma \operatorname{id}_T \in O(T)$ ensures compatibility on the discriminant groups. Thus, $g \oplus h$ extends to a Hodge isometry of $H^2(X, \mathbb{Z})$. Since $g$ preserves the ample cone, the global Torelli theorem implies this isometry is induced by an automorphism of $X$.
\end{proof}

\begin{corollary}\label{Bianchicontainment}

Let $X$ be as above. An element $g \in \operatorname{Bi}_K$ is induced by an automorphism of $X$ if and only if $g \in \operatorname{Bi}_K(2N)$ and the induced action of $g$ on $A_S$ is multiplication by a single sign $\pm 1$.
\end{corollary}

\begin{proof}
Let $g \in \operatorname{Bi}_K$ be represented by a matrix $M \in \operatorname{GL}_2(\mathcal{O}_K)$. The Bianchi action on $S_K$ is given by $$H \longmapsto MHM^\ast.$$ 
If $H$ is positive definite Hermitian, $MHM^\ast$ is also positive definite Hermitian. Hence, $g$ preserves the connected component $C_S^+$ of the positive cone containing $I_2$. Since $S_K(2N)$ has no $(-2)$-classes, $C_S^+$ is exactly the ample cone, meaning $g$ always preserves it. By Proposition~\ref{prop:level-containment}, $g$ is induced by an automorphism of $X$ if and only if its action on $A_S$ is a sign $\pm 1$.

It remains to show that this sign action forces $g \in \operatorname{Bi}_K(2N)$ (the converse is trivial). Inside $A_S$ we have the level subgroup $S_K/(2N)S_K \hookrightarrow A_S$ given by $x \mapsto \frac{x}{2N}$. Under the Hermitian model, this subgroup is naturally identified with $\operatorname{Herm}_2(R_{2N})$, where $R_{2N} := \mathcal{O}_K/2N\mathcal{O}_K$. 

Let $\overline{M}$ denote the reduction of $M$ modulo $2N$. Since $g$ acts on this visible level subgroup by some sign $\epsilon \in \{\pm 1\}$, we have $\overline{M} H \overline{M}^\ast = \epsilon H$ for every $H \in \operatorname{Herm}_2(R_{2N})$.

Applying this identity to the standard basis elements of $\operatorname{Herm}_2(R_{2N})$:
\begin{itemize}
    \item For $e = \begin{pmatrix} 1 & 0 \\ 0 & 0 \end{pmatrix}$, we get $aa^\ast = \epsilon$ (so $a \in R_{2N}^\times$) and $ca^\ast = 0$ (so $c=0$).
    \item For $f = \begin{pmatrix} 0 & 0 \\ 0 & 1 \end{pmatrix}$, we get $b=0$ and $dd^\ast = \epsilon$.
    \item For $P(1) = \begin{pmatrix} 0 & 1 \\ 1 & 0 \end{pmatrix}$, we get $ad^\ast = \epsilon$. 
\end{itemize}

Together with $aa^\ast = \epsilon$, the last condition forces $d=a$. Hence $\overline{M} = aI_2$. The reduction of $g$ modulo $2N$ is therefore projectively trivial, concluding that $g \in \operatorname{Bi}_K(2N)$.
\end{proof}

Now using the expression of Lemma~\ref{lem:dual_generators_Hermitian_lattice}, we will justify why no Galois-conjugation coset element is induced by an automorphism.

\begin{lemma}
\label{lem:galois-coset-exclusion}
Let $K$ be an imaginary quadratic field with ring of integers $\mathcal{O}_K$, and let $N\geq 1$ be odd. Let $X=X_{K,2N}$ be a very general $S_K(2N)$-polarized K3 surface. Let $\sigma$ be the nontrivial Galois automorphism of $K/\mathbb{Q}$, acting on $S_K=\operatorname{Herm}_2(\mathcal{O}_K)$ by $\sigma(P(z))=P(\overline{z})$. Then no element of the coset $\operatorname{Bi}_K\sigma$ is induced by an automorphism of $X$.
\end{lemma}
\begin{proof}
Suppose, for a contradiction, that an element of the form $M\sigma$, with $M\in \operatorname{GL}_2(\mathcal{O}_K)$, is induced by an automorphism of $X$. Since $X$ is very general, the only Hodge isometries of the transcendental lattice are $\pm \operatorname{id}$. Hence the action of $M\sigma$ on the discriminant group $A_{S_K(2N)}$ must be multiplication by a single sign $\varepsilon\in\{\pm 1\}$.

In particular, on the visible level subgroup $S_K/(2N)S_K \subset A_{S_K(2N)}$, the induced action must also be multiplication by $\varepsilon$. Thus, for every $H\in \operatorname{Herm}_2(\mathcal{O}_K/2N\mathcal{O}_K)$, we have $M\sigma(H)M^\ast \equiv \varepsilon H \pmod{2N\mathcal{O}_K}$.

Applying this congruence to $e$, $f$, and $P(1)$, which are fixed by $\sigma$, gives the same calculation as in Proposition~\ref{prop:level-containment}: there exists a unit $u\in(\mathcal{O}_K/2N\mathcal{O}_K)^\times$ such that $M\equiv uI_2 \pmod{2N\mathcal{O}_K}$ and $u\overline{u}\equiv \varepsilon \pmod{2N}$.

Now choose an integral basis $1,\tau$ of $\mathcal{O}_K$. We distinguish the two cases.

First suppose that the fundamental discriminant is odd, so that $\tau=\frac{1+\sqrt{-D}}{2}$ and $\overline{\tau}=1-\tau$. Applying the above congruence to $P(\tau)$, we obtain $M\sigma(P(\tau))M^\ast \equiv u\overline{u}P(\overline{\tau}) \equiv \varepsilon P(1-\tau) \pmod{2N}$. But the global sign condition requires this to be congruent to $\varepsilon P(\tau)$ modulo $2N$. Hence $P(1-2\tau)\equiv 0 \pmod{2N}$. This is impossible, because $1-2\tau\notin 2N\mathcal{O}_K$.

Next suppose that the fundamental discriminant is even. Write $\tau=\frac{\sqrt{-D}}{2}$ and $\overline{\tau}=-\tau$. If $N>1$, the same argument gives $P(-\tau)\equiv P(\tau)\pmod{2N}$, or equivalently $2\tau\in 2N\mathcal{O}_K$, which is impossible.

It remains to treat the case $N=1$. In this case, the level subgroup alone does not give a contradiction, so we use the two-primary discriminant classes. The class $P(1)/4\in A_{S_K(2)}$ gives $MP(1)M^\ast\equiv \varepsilon P(1)\pmod{4S_K}$. Writing $M= \left(\begin{smallmatrix} a & b \\ c & d \end{smallmatrix}\right)$ and using $M\equiv uI_2\pmod{2\mathcal{O}_K}$, we get $a\overline{d}+b\overline{c}\equiv \varepsilon \pmod{4\mathcal{O}_K}$. Since $b,c\in 2\mathcal{O}_K$, this implies $a\overline{d}\equiv \varepsilon \pmod{4\mathcal{O}_K}$.

On the other hand, the class $P(\tau)/D\in A_{S_K(2)}$ gives $M\sigma(P(\tau))M^\ast \equiv \varepsilon P(\tau) \pmod{DS_K}$. Since $\sigma(P(\tau))=P(-\tau)$, the off-diagonal entry of $M\sigma(P(\tau))M^\ast$ is congruent to $-\tau a\overline{d} \equiv -\varepsilon\tau \pmod{4\tau\mathcal{O}_K}$. The global sign condition requires this to be congruent to $\varepsilon\tau$ modulo $D\mathcal{O}_K$. Hence $2\tau\in D\mathcal{O}_K+4\tau\mathcal{O}_K$. Since $D=-4\tau^2$, we have $D\mathcal{O}_K\subset 4\tau\mathcal{O}_K$. Therefore $2\tau\in 4\tau\mathcal{O}_K$, which forces $1\in 2\mathcal{O}_K$, a contradiction.

Thus no element of $\operatorname{Bi}_K\sigma$ can be induced by an automorphism of $X$.
\end{proof}

We now compute the necessary local characters. The discriminant group $A_{S(2N)} = \frac{1}{2N}S^\vee/S$ sits in an exact sequence:
\[ 0 \longrightarrow S/(2N)S \longrightarrow A_{S(2N)} \longrightarrow A_S \longrightarrow 0. \]
The global sign condition thus has two sources: the level part $S/(2N)S$, and the original discriminant group $A_S$. First consider the primes dividing $N$. For an integer $m \geq 1$, put $\mathcal{N}_m := \{ u\overline{u} \mid u \in R_m^\times \} \subset (\mathbb{Z}/m\mathbb{Z})^\times$.

\begin{lemma}
\label{lem:level-norm-character}
Let $g \in \operatorname{Bi}_K(2N)$, and let $p^e \Vert N$ meaning that \(p^e\) exactly divides
\(N\). Choose a representative $M \in \operatorname{GL}_2(\mathcal{O}_K)$. Then there exists $u_{p^e} \in (\mathcal{O}_K/p^e\mathcal{O}_K)^\times$ such that $M \equiv u_{p^e}I_2 \pmod{p^e\mathcal{O}_K}$. The element $\alpha_{p^e}(g) := u_{p^e}\overline{u}_{p^e} \in \mathcal{N}_{p^e}$ is independent of the representative $M$. Moreover, on $S/p^eS \cong \operatorname{Herm}_2(\mathcal{O}_K/p^e\mathcal{O}_K)$, $g$ acts by scalar multiplication with $\alpha_{p^e}(g)$.
\end{lemma}

\begin{proof}
Since $g \in \operatorname{Bi}_K(2N)$, its reduction modulo $p^e\mathcal{O}_K$ is projectively scalar, so $M \equiv u_{p^e}I_2 \pmod{p^e\mathcal{O}_K}$. For any $H \in \operatorname{Herm}_2(\mathcal{O}_K/p^e\mathcal{O}_K)$, $MHM^\ast \equiv (u_{p^e}I_2)H(\overline{u}_{p^e}I_2) = u_{p^e}\overline{u}_{p^e}H \pmod{p^e}$. Replacing $M$ by $\lambda M$ for $\lambda \in \mathcal{O}_K^\times$ scales $u_{p^e}$ by $\lambda$. Since $\lambda\overline{\lambda} = 1$, the norm $u_{p^e}\overline{u}_{p^e}$ is well-defined.
\end{proof}

Next, consider $A_S$. Its odd primary part only appears at odd primes ramified in $K$, where the action is controlled by the determinant.

\begin{lemma} \label{lem:odd-ramified-character}
Let $\ell\neq 2$ be a prime dividing $D$, and write $\ell\mathcal{O}_K=\mathfrak{p}^2$. Let $\delta=\sqrt{-D}$. Then the $\ell$-primary part $A_{S_K,\ell}$ is generated by $u_\ell:=\frac{1}{\ell}P(\delta)\in S_K^\vee/S_K$. For $g=[M]\in \operatorname{Bi}_K$, with $M\in \operatorname{GL}_2(\mathcal{O}_K)$, the induced action of $g$ on $A_{S_K,\ell}$ is multiplication by $\varepsilon_\ell(g) := \det(M)\bmod \mathfrak{p} \in (\mathcal{O}_K/\mathfrak{p})^\times \cong \mathbb{F}_\ell^\times$. Moreover, this scalar is a sign: $\varepsilon_\ell(g)\in\{\pm 1\}$.
\end{lemma}

\begin{proof}
Write $M= \left(\begin{smallmatrix} a & b \\ c & d \end{smallmatrix}\right)$. A direct calculation gives $MP(\delta)M^\ast = \operatorname{Tr}_{K/\mathbb{Q}}(a\delta\overline{b})e + \operatorname{Tr}_{K/\mathbb{Q}}(c\delta\overline{d})f + P(\delta(a\overline{d}-b\overline{c}))$. For every $x\in\mathcal{O}_K$, we have $\operatorname{Tr}_{K/\mathbb{Q}}(\delta x) = \delta(x-\overline{x}) \in D\mathbb{Z} \subset \ell\mathbb{Z}$. Hence the diagonal terms vanish modulo $\ell S_K$. Since $\ell$ is ramified, complex conjugation acts trivially on $\mathcal{O}_K/\mathfrak{p}$. Therefore $a\overline{d}-b\overline{c} \equiv ad-bc = \det(M) \pmod{\mathfrak{p}}$. Because $\delta\in \mathfrak{p}$, this implies $\delta\bigl(a\overline{d}-b\overline{c}-\det(M)\bigr) \in \mathfrak{p}^2 = \ell\mathcal{O}_K$. Thus $MP(\delta)M^\ast \equiv P(\delta\det(M)) \pmod{\ell S_K}$. Dividing by $\ell$, we obtain $g(u_\ell) = \varepsilon_\ell(g)u_\ell$, where $\varepsilon_\ell(g)=\det(M)\bmod\mathfrak{p}$.

It remains to check that $\varepsilon_\ell(g)$ is well-defined and is a sign. If $M$ is replaced by $\lambda M$, with $\lambda\in\mathcal{O}_K^\times$, then $\det(\lambda M)=\lambda^2\det(M)$. For an odd ramified prime $\ell$, the reduction of every global unit satisfies $\lambda^2\equiv 1\pmod{\mathfrak{p}}$. Indeed, if $\mathcal{O}_K^\times=\{\pm 1\}$, this is immediate. The only imaginary quadratic field with extra units and an odd ramified prime is $K=\mathbb{Q}(\sqrt{-3})$, where $\ell=3$ and $(\mathcal{O}_K/\mathfrak{p})^\times\cong \mathbb{F}_3^\times=\{\pm 1\}$. Therefore the scalar $\varepsilon_\ell(g)$ is independent of the representative $M$. Finally, since $M\in \operatorname{GL}_2(\mathcal{O}_K)$, we have $\det(M)\in \mathcal{O}_K^\times$. By the preceding discussion, its reduction modulo $\mathfrak{p}$ is a sign. Hence $\varepsilon_\ell(g)\in\{\pm 1\}$.
\end{proof}

We now package these local scalars. Define
\[ R_N := \left( \prod_{\ell \mid D, \ell \neq 2} \{\pm 1\} \right) \times \left( \prod_{p^e \Vert N} \mathcal{N}_{p^e} \right), \]
and let $\Psi_N \colon \operatorname{Bi}_K(2N) \longrightarrow R_N$ be given by $\Psi_N(g) := \big( (\epsilon_\ell(g))_{\ell \mid D, \ell \neq 2}, (\alpha_{p^e}(g))_{p^e \Vert N} \big)$.
The diagonal sign subgroup is $\Delta_N := \big\{ \big( (\sigma), (\sigma) \big) \mid \sigma \in \{\pm 1\}, \sigma \in \mathcal{N}_{p^e} \big\}$.

\begin{definition}
\label{def:CN}
Let $\mathcal{C}_N \subset \operatorname{Bi}_K(2N)$ be the subgroup of elements $g$ for which there exists a single sign $\sigma \in \{\pm 1\}$ such that:
\begin{enumerate}
    \item $\alpha_{p^e}(g) = \sigma$ for every $p^e \Vert N$;
    \item $\epsilon_\ell(g) = \sigma$ for every odd prime $\ell \mid D$;
    \item If $2 \mid D$, the induced action of $g$ on $A_{S_K(2N),2}$ is multiplication by $\sigma$.
\end{enumerate}
\end{definition}

\begin{theorem}
\label{thm:main-bianchi-aut-group}
Let $K$ be an imaginary quadratic field with fundamental discriminant $-D$. Let $N$ be an odd positive integer, and assume that no prime divisor of $N$ ramifies in $K$. Let $X = X_{K,2N}$ be very general with $\operatorname{NS}(X) \cong S_K(2N)$. Then $\operatorname{Bi}_K \cap \operatorname{Aut}(X) = \mathcal{C}_N$. In particular, if $2 \nmid D$, then $\operatorname{Bi}_K \cap \operatorname{Aut}(X) = \Psi_N^{-1}(\Delta_N)$.
\end{theorem}

\begin{proof}
By Proposition~\ref{prop:level-containment}, $g \in \operatorname{Bi}_K$ is induced by an automorphism of $X$ if and only if $g \in \operatorname{Bi}_K(2N)$ and $g$ acts on $A_{S_K(2N)}$ by a single sign $\sigma \in \{\pm 1\}$. For $p^e \parallel N$, $p$ is unramified in $K$, so $S_K \otimes \mathbb{Z}_p$ is unimodular; $g$ acts on $A_{S_K(2N),p}$ by $\alpha_{p^e}(g)$, forcing $\alpha_{p^e}(g) = \sigma$. For an odd prime $\ell \mid D$, $\ell \nmid N$, so $A_{S_K(2N),\ell} \cong A_{S_K,\ell}$, forcing $\epsilon_{\ell}(g) = \sigma$. If $2 \mid D$, the condition on $A_{S_K(2N),2}$ completes the definition of $\mathcal{C}_N$. If $2 \nmid D$, $S_K \otimes \mathbb{Z}_2$ is unimodular; since $g \in \operatorname{Bi}_K(2N)$, its $2$-primary action is trivial, and $+1 \equiv -1 \pmod 2$ poses no obstruction. The condition then reduces exactly to $\Psi_N(g) \in \Delta_N$.
\end{proof}

\begin{corollary}
\label{cor:index-bound}
Let $K$ be an imaginary quadratic field with fundamental discriminant
$-D$ such that $2\nmid D$. Let $N$ be an odd positive integer, and
assume that no prime divisor of $N$ ramifies in $K$. Let
$X=X_{K,2N}$ be very general with $\operatorname{NS}(X)\cong S_K(2N).$
Then
\[
    [\operatorname{Bi}_K(2N):
    \operatorname{Bi}_K\cap\operatorname{Aut}(X)]
    =
    \frac{|\operatorname{im}\Psi_N|}
    {|\operatorname{im}\Psi_N\cap\Delta_N|}.
\]
Moreover,
\[
    [\operatorname{Bi}_K(2N):
    \operatorname{Bi}_K\cap\operatorname{Aut}(X)]
    \leq
    2^{\omega(N)},
\]
where $\omega(N)$ denotes the number of distinct prime divisors of $N$.
\end{corollary}

\begin{proof}
By Theorem~\ref{thm:main-bianchi-aut-group}, $\operatorname{Bi}_K\cap\operatorname{Aut}(X)
    =
    \Psi_N^{-1}(\Delta_N).$
The First Isomorphism Theorem therefore gives
\[
    [\operatorname{Bi}_K(2N):
    \operatorname{Bi}_K\cap\operatorname{Aut}(X)]
    =
    \frac{|\operatorname{im}\Psi_N|}
    {|\operatorname{im}\Psi_N\cap\Delta_N|}.
\]

We now describe a subgroup containing
$\operatorname{im}\Psi_N$. For
$g=[M]\in\operatorname{Bi}_K(2N)$ and $p^e\Vert N$, choose
$u_{p^e}\in(\mathcal O_K/p^e\mathcal O_K)^\times$ such that $M\equiv u_{p^e}I_2
    \pmod{p^e\mathcal O_K}.$
Taking determinants gives $u_{p^e}^2
    \equiv
    \det(M)
    \pmod{p^e\mathcal O_K}.$
Hence
\[
    \alpha_{p^e}(g)^2
    =
    (u_{p^e}\overline{u}_{p^e})^2
    \equiv
    \det(M)\overline{\det(M)}
    =
    1
    \pmod{p^e}.
\]
Since $p$ is odd, the only solutions of $x^2=1$ in
$(\mathbb Z/p^e\mathbb Z)^\times$ are $\pm1$. Therefore $\alpha_{p^e}(g)\in\{\pm1\}.$

On the other hand, Lemma~\ref{lem:odd-ramified-character} gives $\epsilon_\ell(g)
    =
    \det(M)\pmod{\mathfrak p_\ell}$ for every odd prime $\ell\mid D$. If
$K\neq\mathbb Q(\sqrt{-3})$, then
$\mathcal O_K^\times=\{\pm1\}$, so all the values
$\epsilon_\ell(g)$ are equal to a single sign. If
$K=\mathbb Q(\sqrt{-3})$, then $D=3$, so there is only one odd
ramified prime and the same conclusion is automatic.

It follows that $\operatorname{im}\Psi_N\subset H_N,$
where
\[
    H_N
    :=
    \left\{
        \left(
            (\eta)_{\substack{\ell\mid D\\ \ell\neq 2}},
            (a_{p^e})_{p^e\Vert N}
        \right)
        \,\middle|\,
        \eta,a_{p^e}\in\{\pm1\}
    \right\}.
\]
The diagonal sign subgroup $\Delta_N$ is contained in $H_N$, and $[H_N:\Delta_N]=2^{\omega(N)}.$
Writing $I_N:=\operatorname{im}\Psi_N$, we obtain
\[
    \frac{|I_N|}{|I_N\cap\Delta_N|}
    =
    [I_N\Delta_N:\Delta_N]
    \leq
    [H_N:\Delta_N]
    =
    2^{\omega(N)}.
\]
This proves the assertion.
\end{proof}

\begin{corollary}
\label{cor:class-number-one-level-two}
Let $p \equiv 3 \pmod 4$ be a prime, and let $K = \mathbb{Q}(\sqrt{-p})$. Let $X_{K,2}$ be very general with $\operatorname{NS}(X_{K,2}) \cong S_K(2)$. Then $\operatorname{Aut}(X_{K,2}) \cong \operatorname{Bi}_K(2)\cong P\Gamma_K(2)$.
\end{corollary}

\begin{proof}
Because $p \equiv 3 \pmod 4$, the fundamental discriminant of $K$ is exactly $-p$, which is an odd prime (so $2 \nmid p$). Applying Corollary~\ref{cor:index-bound} with $N=1$, we get $\Bi_K(2)=\Bi_K\cap\Aut(X).$

Furthermore, because the fundamental discriminant $-p$ has only one prime divisor, genus theory implies that the $2$-rank of the ideal class group $\operatorname{Cl}(K)$ is $0$. Thus, $\operatorname{Cl}(K)$ has odd order, meaning the quotient $\operatorname{Cl}(K)/\operatorname{Cl}(K)^2$ is trivial. Consequently, the extended Bianchi group is simply $O^+(S_K) = \operatorname{Bi}_K \rtimes \langle \sigma \rangle$. By Lemma~\ref{lem:galois-coset-exclusion}, the coset $\operatorname{Bi}_K\sigma$ is not geometrically induced, ensuring $\operatorname{Aut}(X_{K,2}) \subseteq \operatorname{Bi}_K$. Injectivity of the lattice representation gives $\operatorname{Aut}(X_{K,2}) = \operatorname{Bi}_K(2)$. Combine with Proposition~\ref{cor:Bi-vs-PGamma-2}, we also got $\Aut(X_{K,2})=P\Gamma(2).$
\end{proof}

We now specialize to very general $S_K(2)$-polarized K3 surfaces where the prime $2$ ramifies in the imaginary quadratic field $K$. The odd level characters disappear, and by Proposition~\ref{prop:level-containment}, we need only determine which isometries satisfy the global sign condition on the discriminant group $A_{S_K(2)}$. We examine three distinct settings where $2$ ramifies: $K = \mathbb{Q}(i)$, $K = \mathbb{Q}(\sqrt{-2})$, and the infinite family $K = \mathbb{Q}(\sqrt{-p})$ for primes $p \equiv 1 \pmod 4$.
\begin{theorem}\label{thm:ramified-level-two}
Let $K\in\{\mathbb{Q}(i),\mathbb{Q}(\sqrt{-2})\}$, or let $K=\mathbb{Q}(\sqrt{-p})$ for an odd prime $p\equiv 1\pmod 4$. Let $X_{K,2}$ be a very general $S_K(2)$-polarized K3 surface. The natural representation $\rho\colon\operatorname{Aut}(X_{K,2})\to O(S_K(2))$ is injective, and its image is precisely the principal congruence subgroup, so that:
$\operatorname{Aut}(X_{K,2})\,\cong P\Gamma_K(2).$
\end{theorem}

We evaluate which isometries of the lattice $S_K(2)$ are induced by geometric automorphisms.
The proof of the theorem relies on the following three lemmas.

\begin{lemma}\label{lem:pgamma-sign}
Every element of $P\Gamma_K(2)$ acts on the discriminant group $A_{S_K(2)}$ by a global sign. Consequently, $P\Gamma_K(2)\subseteq \Aut(X_{K,2})$.
\end{lemma}
\begin{proof}
Let $[M]\in P\Gamma_K(2)$, and choose a representative $M=\left(\begin{smallmatrix}a&b\\c&d\end{smallmatrix}\right)\in\operatorname{GL}_2(\mathcal{O}_K)$ such that $M\equiv I_2\pmod{2\mathcal{O}_K}$. This congruence implies $\det(M)\equiv 1\pmod{2\mathcal{O}_K}$. 
For $K\in\{\mathbb{Q}(\sqrt{-2}),\mathbb{Q}(\sqrt{-p})\}$, the global units are $\{\pm 1\}$, which both reduce to $1$ modulo $2$. For $K=\mathbb{Q}(i)$, the global units are $\{\pm 1,\pm i\}$; since $\pm i\not\equiv 1\pmod{2\mathbb{Z}[i]}$, the determinant congruence forces $\det(M)\in\{\pm 1\}$. Thus, in all three cases, \(\det(M)\in\{\pm1\}\); we set
\(\delta:=\det(M)\).

We verify the action $H\mapsto MHM^\ast$ on the generators of $A_{S_K(2)}$. Since $M\equiv I_2\pmod{2\mathcal{O}_K}$, the action of $M$ on $S_K/2S_K$ is trivial. Therefore the classes $e/2$ and $f/2$ are fixed in $A_{S_K(2)}$. Since they are $2$-torsion, multiplication by $+1$ and by $-1$ agree on them. For the discriminant generators, we separate the cases:
\begin{itemize}
    \item \emph{Generator $P(1)$ (all fields):} The generator is $P(1)/4$, so we work modulo $4S_K$. The off-diagonal entry of $MP(1)M^\ast$ is $a\overline{d}+b\overline{c}$. We evaluate the difference from $\delta$:
    $$a\overline{d}+b\overline{c}-\delta=a\overline{d}+b\overline{c}-(ad-bc)=a(\overline{d}-d)+b(\overline{c}+c).$$
    Because $a,d\equiv 1\pmod{2\mathcal{O}_K}$ and $b,c\in 2\mathcal{O}_K$, both terms are divisible by $4$. The diagonal entries $a\overline{b}+b\overline{a}$ and $c\overline{d}+d\overline{c}$ are multiples of $4$ by the same parity argument. Thus, $MP(1)M^\ast\equiv\delta P(1)\pmod{4S_K}$.
    \item \emph{Generator $P(i)$ ($K=\mathbb{Q}(i)$):} Working modulo $4S_K$, the off-diagonal entry of $MP(i)M^\ast$ is $i(a\overline{d}-b\overline{c})$. We check the difference:
    $$i(a\overline{d}-b\overline{c})-i\delta=i\left(a(\overline{d}-d)-b(\overline{c}-c)\right).$$
    The term in parentheses is divisible by $4$, and the diagonal entries are similarly divisible by $4$. Thus, $MP(i)M^\ast\equiv\delta P(i)\pmod{4S_K}$.
    \item \emph{Generator $P(\tau)$ ($K\in\{\mathbb{Q}(\sqrt{-2}),\mathbb{Q}(\sqrt{-p})\}$):} The generator is $P(\tau)/D$, working modulo $DS_K$ (where $D=8$ or $D=4p$). The off-diagonal entry of $MP(\tau)M^\ast$ is $\tau(a\overline{d}-b\overline{c})$. The difference is:
    $$\tau(a\overline{d}-b\overline{c})-\delta\tau=\tau\left(a(\overline{d}-d)-b(\overline{c}-c)\right).$$
    The parenthetical term is divisible by $4\tau$, making the entire expression divisible by $4\tau^2=-D$. The diagonal entries have the form $\tau(a\overline{b}-b\overline{a})$ and $\tau(c\overline{d}-d\overline{c})$. Since $b,c\in 2\mathcal{O}_K$ and $a,d\equiv 1\pmod{2\mathcal{O}_K}$, the expressions $a\overline{b}-b\overline{a}$ and $c\overline{d}-d\overline{c}$ are divisible by $4\tau$. Multiplying by $\tau$, the diagonal entries are divisible by $4\tau^2=-D$. Hence the diagonal entries vanish modulo $DS_K$. Thus, $MP(\tau)M^\ast\equiv\delta P(\tau)\pmod{DS_K}$.
\end{itemize}
Consequently, every element of $P\Gamma_K(2)$ acts on $A_{S_K(2)}$ by the global sign $\delta$.
\end{proof}

\begin{lemma}\label{lem:scalar-defect}
The nontrivial scalar coset in $\operatorname{Bi}_K(2)\setminus P\Gamma_K(2)$ fails the global sign condition.
\end{lemma}
\begin{proof}
We evaluate the index $[\operatorname{Bi}_K(2):P\Gamma_K(2)]$. An element $[M]\in\operatorname{Bi}_K(2)$ admits a representative $M\equiv\lambda I_2\pmod{2\mathcal{O}_K}$ for some local unit $\lambda\in(\mathcal{O}_K/2\mathcal{O}_K)^\times$.
\begin{itemize}
    \item \emph{For $K=\mathbb{Q}(i)$:} The local unit group is $(\mathcal{O}_K/2\mathcal{O}_K)^\times=\{1,i\}$. The reduction map from the global units $\mathcal{O}_K^\times=\{\pm 1,\pm i\}$ to $\{1,i\}$ is surjective. We can choose a global unit $u\equiv\lambda\pmod 2$. The representative $u^{-1}M$ satisfies $u^{-1}M\equiv I_2\pmod 2$, meaning $[M]\in P\Gamma_K(2)$. Hence, $\operatorname{Bi}_K(2)=P\Gamma_K(2)$.
    \item \emph{For $K=\mathbb{Q}(\sqrt{-2})$:} The quotient ring is $\mathbb{F}_2[\tau]/(\tau^2)$ with unit group $\{1,\tau+1\}$. Global units $\{\pm 1\}$ both reduce to $1$, yielding $[\operatorname{Bi}_K(2):P\Gamma_K(2)]\leq 2$.
    \item \emph{For $K=\mathbb{Q}(\sqrt{-p})$:} Since $\tau^2=-p\equiv 1\pmod 2$, the quotient ring is $\mathcal{O}_K/2\mathcal{O}_K\cong\mathbb{F}_2[\epsilon]/(\epsilon^2)$ where $\epsilon=\tau+1$. The local unit group is $\{1,\tau\}$. Global units map only to $1$, again yielding $[\operatorname{Bi}_K(2):P\Gamma_K(2)]\leq 2$.
\end{itemize}
Since $\Aut(X_{K,2})$ is a subgroup containing $P\Gamma_K(2)$, and the index is at most $2$, it suffices to test a single representative of the nontrivial coset $\operatorname{Bi}_K(2)\setminus P\Gamma_K(2)$ for the latter two fields.
\begin{itemize}
    \item \emph{For $K=\mathbb{Q}(\sqrt{-2})$:} Let $M_0=\left(\begin{smallmatrix}1+\tau&2\\2&1-\tau\end{smallmatrix}\right)$. We have $\det(M_0)=-1$. Because $M_0\equiv(1+\tau)I_2\pmod{2\mathcal{O}_K}$, it reduces to the nontrivial scalar, confirming the index is exactly $2$. Evaluating its action on $P(1)/4$ modulo $4S_K$ yields:
    $$M_0P(1)M_0^\ast=\begin{pmatrix}4&3+2\tau\\3-2\tau&4\end{pmatrix}.$$
    The off-diagonal entry requires $(3+2\tau)\mp 1\in 4\mathcal{O}_K$. However, $2+2\tau\notin 4\mathcal{O}_K$ and $4+2\tau\notin 4\mathcal{O}_K$. Thus, $M_0$ fails the global sign condition.
    \item \emph{For $K=\mathbb{Q}(\sqrt{-p})$:} Let $s=(p-1)/4$ and $M_0=\left(\begin{smallmatrix}\tau&2\\-2s&\tau\end{smallmatrix}\right)$. We have $\det(M_0)=-1$. Again, $M_0$ reduces to the nontrivial scalar $\tau I_2\pmod{2\mathcal{O}_K}$, confirming the index is exactly $2$. Calculation gives $M_0P(1)M_0^\ast=P(1)$, which requires the global sign to be $+1$. However, testing $P(\tau)$ modulo $4pS_K$ gives:
    $$M_0P(\tau)M_0^\ast=\begin{pmatrix}-4p&(2p-1)\tau\\-(2p-1)\tau&-p(p-1)\end{pmatrix}.$$
    This requires $(2p-1)\tau-\tau=2(p-1)\tau\in 4p\mathcal{O}_K$, which is false.
\end{itemize}
\end{proof}

\begin{lemma}\label{lem:extended-coset}
No element of the extended ideal-class coset $\widehat{\operatorname{Bi}}_K\setminus\operatorname{Bi}_K$ acts on $A_{S_K(2)}$ by a global sign. Furthermore, the same calculation excludes the extended Galois coset $(\widehat{\operatorname{Bi}}_K\setminus\operatorname{Bi}_K)\sigma$.
\end{lemma}
\begin{proof}
For $K=\mathbb{Q}(i)$ and $K=\mathbb{Q}(\sqrt{-2})$, the class number is $1$ and the discriminant yields no nontrivial $2$-torsion ideal classes, so $\widehat{\operatorname{Bi}}_K=\operatorname{Bi}_K$. Using the description of Example~\ref{exampleextemdedBianchi}. Computing its action on $P(1)=\left(\begin{smallmatrix}0&1\\1&0\end{smallmatrix}\right)$ directly:
$$W_2P(1)W_2^\ast=\begin{pmatrix}1+\tau&2\\2s&1-\tau\end{pmatrix}\begin{pmatrix}0&1\\1&0\end{pmatrix}\begin{pmatrix}1-\tau&2s\\2&1+\tau\end{pmatrix}=\begin{pmatrix}4&2\tau\\-2\tau&4s\end{pmatrix}.$$
Dividing by $2$ and reducing modulo $2\mathcal{O}_K$, we obtain:
$$\frac{1}{2}W_2P(1)W_2^\ast=\begin{pmatrix}2&\tau\\-\tau&2s\end{pmatrix}\equiv\begin{pmatrix}0&\tau\\\tau&0\end{pmatrix}=P(\tau)\pmod{2\mathcal{O}_K}.$$
Suppose an element of this extended coset is induced by an automorphism of $X_{K,2}$. It must have the form $MW_2$ for some ordinary Bianchi matrix $M=\left(\begin{smallmatrix}a&b\\c&d\end{smallmatrix}\right)\in\operatorname{GL}_2(\mathcal{O}_K)$. By Proposition~\ref{prop:level-containment}, $MW_2$ must act on $A_{S_K(2)}$ by a global sign $\pm 1$. Because $+1\equiv -1\pmod 2$, the induced action of $MW_2$ on the level-$2$ subgroup $S_K/2S_K$ must be the identity.
Its action on $P(1)$ must yield $P(1)$ modulo $2\mathcal{O}_K$. Composing the actions, we require:
$$M\left(\frac{1}{2}W_2P(1)W_2^\ast\right)M^\ast\equiv P(1)\pmod{2\mathcal{O}_K},$$
which simplifies to $MP(\tau)M^\ast\equiv P(1)\pmod{2\mathcal{O}_K}$. We compute the left-hand side directly:
$$MP(\tau)M^\ast\equiv\begin{pmatrix}a&b\\c&d\end{pmatrix}\begin{pmatrix}0&\tau\\\tau&0\end{pmatrix}\begin{pmatrix}\overline{a}&\overline{c}\\\overline{b}&\overline{d}\end{pmatrix}=\begin{pmatrix}\tau(a\overline{b}+b\overline{a})&\tau(a\overline{d}+b\overline{c})\\\tau(c\overline{b}+d\overline{a})&\tau(c\overline{d}+d\overline{c})\end{pmatrix}\pmod{2\mathcal{O}_K}.$$
Since $2\equiv 0\pmod{2\mathcal{O}_K}$, the diagonal entries vanish. For the off-diagonal entry, complex conjugation is trivial modulo $2$, so $a\overline{d}+b\overline{c}\equiv ad-bc=\det(M)\pmod{2\mathcal{O}_K}$. Thus, that congruence requires:
$$\begin{pmatrix}0&\tau\det(M)\\\tau\det(M)&0\end{pmatrix}\equiv\begin{pmatrix}0&1\\1&0\end{pmatrix}\pmod{2\mathcal{O}_K},$$
which implies $\tau\det(M)\equiv 1\pmod{2\mathcal{O}_K}$.
Since $\det(M)\in\{\pm 1\}$ and both signs reduce to $1$ modulo $2$, this forces $\tau\equiv 1\pmod{2\mathcal{O}_K}$. However, this requires $\tau-1\in 2\mathcal{O}_K$. Equivalently, in the quotient ring $\mathcal{O}_K/2\mathcal{O}_K\cong\mathbb{F}_2[\epsilon]/(\epsilon^2)$, the class $\epsilon=\tau+1=\tau-1$ must be zero, which is false. Thus, no such matrix $M$ can exist. Finally, since $\sigma(P(1))=P(1)$, we have $\frac{1}{2}W_2\sigma(P(1))W_2^\ast=\frac{1}{2}W_2P(1)W_2^\ast\equiv P(\tau)\pmod{2\mathcal{O}_K}$. Therefore the same contradiction $\tau\det(M)\equiv 1\pmod{2\mathcal{O}_K}$ follows, excluding the extended Galois coset as well.
\end{proof}

\begin{proof}[Proof of Theorem~\ref{thm:ramified-level-two}]
By Lemma~\ref{lem-twoelementary}, the representation $\rho$ is injective because the discriminant group $A_{S_K(2)}$ is not $2$-elementary in all cases considered (it contains $4$-torsion for $K=\mathbb{Q}(i)$, $8$-torsion for $K=\mathbb{Q}(\sqrt{-2})$, and $4p$-torsion for $K=\mathbb{Q}(\sqrt{-p})$). By the Torelli theorem, since the lattice $S_K(2)$ contains no $(-2)$-classes, the set of roots is empty ($\Delta(S_K(2))=\emptyset$) and the Weyl group is trivial ($W(S_K(2))=1$). Consequently, the ample cone coincides with the positive cone.

We have decomposed the base isometries into the ordinary Bianchi group, the extended ideal-class coset, and the Galois cosets. Lemma~\ref{lem:galois-coset-exclusion} rules out the ordinary Galois coset, and Lemma~\ref{lem:extended-coset} rules out both the extended ideal-class coset and the extended Galois coset. We are therefore restricted to the ordinary Bianchi group $\operatorname{Bi}_K$. By Proposition~\ref{prop:level-containment}, the global sign condition forces any such element to belong to $\operatorname{Bi}_K(2)$. By Lemmas~\ref{lem:pgamma-sign} and~\ref{lem:scalar-defect}, the elements of $\operatorname{Bi}_K(2)$ satisfying the global sign condition are precisely those in the principal congruence subgroup $P\Gamma_K(2)$. Thus, $\rho(\operatorname{Aut}(X_{K,2}))=P\Gamma_K(2)$. Since $\rho$ is injective, this gives $\operatorname{Aut}(X_{K,2})\cong P\Gamma_K(2)$.
\end{proof}

We now turn to the study of higher-level congruence subgroups. The coincidence $\operatorname{Aut}(X_{K,2})\cong P\Gamma_K(2)$ observed at level $2$ might naturally suggest that the principal congruence subgroup $P\Gamma_K(2N)$ rather than the Bianchi congruence subgroup $\operatorname{Bi}_K(2N)$ is always the appropriate group to describe these automorphisms. We conclude this section with examples demonstrating that this assumption is false: at higher levels, the automorphism group is indeed isomorphic to a congruence subgroup, but not the principal one.

\begin{theorem}
\label{prop:split-prime-power-principal-congruence}
Let \(p\) be a prime such that \(p \equiv 7 \pmod 8\), and let \(K=\mathbb Q(\sqrt{-p})\). Let \(N=q^e\), where \(q \equiv 3 \pmod 4\) is an odd prime that splits in \(K\). Then the automorphism group of level \(2N\) coincides with the principal congruence subgroup of level \(2N\):
    \[
        \operatorname{Aut}(X_{K,2N}) = P\Gamma_K(2N).
    \]
\end{theorem}

\begin{proof}
Since \(2\) splits in \(K\), we have \(\mathcal O_K^\times=\{\pm1\}\). Let \(g\in\mathcal \Aut(X_{K,2N})\). Because \(\Aut(X_{K,2N})\subseteq \operatorname{Bi}_K(2N)\), we can choose a representative \(M\in\operatorname{GL}_2(\mathcal O_K)\) such that \(M\equiv uI_2 \pmod{2N\mathcal O_K}\) for some \(u\in(\mathcal O_K/2N\mathcal O_K)^\times\). Let \(\delta:=\det(M)\in\{\pm1\}\). Because \(q\) splits in \(K\), there is an isomorphism \(\mathcal O_K/q^e\mathcal O_K \cong \mathbb Z/q^e\mathbb Z \times \mathbb Z/q^e\mathbb Z\), under which complex conjugation exchanges the factors. Writing the \(q^e\)-component of \(u\) as \(u_q=(a,b)\), the relation \(\det(M)=\delta\) modulo \(q^e\mathcal O_K\) implies \(a^2=b^2=\delta\) in \(\mathbb Z/q^e\mathbb Z\).

The action of \(g\) on the \(q^e\)-primary level component of the discriminant group is the norm scalar \(\alpha_{q^e}(g) = u_q\overline{u_q} = (a,b)(b,a) = (ab,ab)\), which we identify with \(ab \in (\mathbb Z/q^e\mathbb Z)^\times\). At the same time, \(g\) acts on the original \(p\)-primary part via multiplication by \(\epsilon_p(g)=\delta\). Since \(g\in\mathcal \Aut(X_{K,2N})\), all primary components share a common sign, meaning \(ab=\delta\). Combining this with \(a^2=\delta\) yields \(ab=a^2\), and since \(a\) is a unit, \(a=b\). Thus, \(u_q=(a,a)\) and \(a^2=\delta\).

If \(\delta=-1\), then \(a^2 \equiv -1 \pmod q\). However, \(q\equiv 3 \pmod 4\) makes \(-1\) a quadratic non-residue modulo \(q\), so this is impossible. Therefore, \(\delta=1\). Consequently, \(a^2\equiv 1 \pmod{q^e}\), which yields \(a\equiv\pm1 \pmod{q^e}\) since \(q\) is odd. Thus, \(M\equiv\varepsilon I_2 \pmod{q^e\mathcal O_K}\) for some \(\varepsilon\in\{\pm1\}\).

Modulo \(2\), \(\mathcal O_K/2\mathcal O_K \cong \mathbb F_2\times\mathbb F_2\) has a trivial unit group, forcing \(M\equiv I_2 \pmod{2\mathcal O_K}\) and \(\varepsilon\equiv 1 \pmod{2\mathcal O_K}\). Since \(2\mathcal O_K\) and \(q^e\mathcal O_K\) are relatively prime, the Chinese remainder theorem gives \(\varepsilon^{-1}M\equiv I_2 \pmod{2q^e\mathcal O_K}\). Hence \(g\in P\Gamma_K(2N)\), proving \(\mathcal \Aut(X_{K,2N})\subseteq P\Gamma_K(2N)\).

Conversely, let \(g\in P\Gamma_K(2N)\). It has a representative \(M\equiv I_2 \pmod{2N\mathcal O_K}\). The norm scalar on the \(q^e\)-primary level component is \(\alpha_{q^e}(g)=1\), and \(\det(M)\equiv 1 \pmod{q\mathcal O_K}\). Since \(\det(M)\in\{\pm1\}\) and \(q\) is odd, \(\det(M)=1\), making \(\epsilon_p(g)=1\). Every primary component is acted on by \(+1\), so \(g\in\mathcal C_N\). This proves \(P\Gamma_K(2N)\subseteq\mathcal \Aut(X_{K,2N})\).
\end{proof}

\begin{theorem}
\label{thm:prime-condition-Aut-Bi}
Let $K=\mathbb Q(\sqrt{-p})$, where $p>3$ is a prime satisfying $p\equiv3\pmod4$.
Let $N>1$ be odd, and assume that every prime divisor $\ell\mid N$ is inert in $K$ and congruent to $1$ modulo $4$.

Let $X_{K,2N}$ be a very general $S_K(2N)$-polarized $K3$ surface. Then, under the normalized Bianchi action, the following hold:
\begin{enumerate}
    \item $\operatorname{Aut}(X_{K,2N}) \cong \Bi_K(2N)$.
    \item $P\Gamma_K(2N) \subsetneq \Bi_K(2N)$.
\end{enumerate}
\end{theorem}

\begin{proof}
\textit{Proof of (1).} 
Since $p\equiv3\pmod4$, the fundamental discriminant of $K$ is $D=-p$.
In particular, $D$ is odd. Moreover, $\mathcal O_K^\times=\{\pm1\}$, because $p>3$. Let $g=[M]\in\Bi_K(2N)$.
Since $\det(M)$ is a unit, we may write $\delta:=\det(M)\in\{\pm1\}$.
This sign is independent of the chosen integral representative of the projective class $g$. Fix a prime power $\ell^e\Vert N$. Since $g\in\Bi_K(2N)$, there exists $u_\ell\in \bigl(\mathcal O_K/\ell^e\mathcal O_K\bigr)^\times$ such that $M\equiv u_\ell I_2 \pmod{\ell^e\mathcal O_K}$.
Taking determinants gives $u_\ell^2=\delta$ in $\mathcal O_K/\ell^e\mathcal O_K$.
Conjugating this equality gives $\overline{u}_\ell^{\,2}=\delta$.
Therefore $\left(\frac{\overline{u}_\ell}{u_\ell}\right)^2=1$.
Since $\ell$ is odd and $\mathcal O_K/\ell^e\mathcal O_K$ is a local ring, it follows that $\frac{\overline{u}_\ell}{u_\ell}\in\{\pm1\}$.

We determine this sign after reducing modulo $\ell$. Since $\left(\frac{-p}{\ell}\right)=-1$, the prime $\ell$ is inert in $K$, and hence $\mathcal O_K/\ell\mathcal O_K \cong \mathbb F_{\ell^2}$. Conjugation on this residue field is the Frobenius automorphism.
\begin{enumerate}
    \item If $\delta=1$, then $u_\ell$ is a root of $x^2-1$, so $u_\ell\equiv\pm1\pmod{\ell}$.
Thus $u_\ell$ is fixed by conjugation modulo $\ell$. Since $u_\ell$ and $\overline{u}_\ell$ are roots of the same polynomial and $2u_\ell$ is a unit, the lift is unique, and therefore $\overline{u}_\ell=u_\ell$ modulo $\ell^e$.
\item 
If $\delta=-1$. Then $u_\ell$ is a root of $x^2+1$. Since $\ell\equiv1\pmod4$, the roots of $x^2+1$ belong to the prime field $\mathbb F_\ell$. They are therefore fixed by Frobenius. Again, uniqueness of the Hensel lift gives $\overline{u}_\ell=u_\ell$ modulo $\ell^e$.
\end{enumerate}

Thus, in both cases, $\alpha_{\ell^e}(g) = u_\ell\overline{u}_\ell = u_\ell^2 = \delta$. The only odd ramified prime of $K$ is $p$. By the ramified prime calculation, the corresponding discriminant character is $\epsilon_p(g)=\det(M)=\delta$.
Since $D$ is odd, the $2$-primary condition is automatic.
Therefore all local discriminant characters agree with the single global sign $\sigma=\delta$.
By the global-sign criterion, $g|_{A_{S_K(2N)}} = \delta\operatorname{id}_{A_{S_K(2N)}}$.
Hence every element of $\Bi_K(2N)$ satisfies the criterion of Proposition~\ref{prop:level-containment}.

\medskip
\textit{Proof of (2).}
Set $m:=2N$.
For every prime power $\ell^e\Vert N$, the congruence $x^2\equiv-1\pmod{\ell^e}$ has a solution because $\ell\equiv1\pmod4$. Since $1^2\equiv-1\pmod2$, the Chinese remainder theorem gives an integer $u$ such that $u^2\equiv-1\pmod m$.
Since $m>2$, one has $u\not\equiv\pm1\pmod m$.

Write $s:=\frac{u^2+1}{m}\in\mathbb Z$.
Since $u$ is invertible modulo $m$, choose $t\in\mathbb Z$ such that $ut\equiv-s\pmod m$.
Then $c:=\frac{s+ut}{m}$ is an integer. Define
$M= \begin{pmatrix} u&m\\ mc&u+mt \end{pmatrix}$.
A direct computation gives $\det(M) = -1. $
Thus $M\in\operatorname{GL}_2(\mathbb Z) \subseteq \operatorname{GL}_2(\mathcal O_K)$.
Moreover, $M\equiv uI_2\pmod{m\mathcal O_K}$.
Since $u$ is invertible modulo $m$, the projective reduction of $M$ is trivial, and therefore $[M]\in\Bi_K(2N)$.

Suppose that $[M]\in P\Gamma_K(2N)$. Then some global unit $\varepsilon\in\mathcal O_K^\times=\{\pm1\}$ satisfies $\varepsilon M\equiv I_2 \pmod{m\mathcal O_K}$.
Since $M\equiv uI_2$, this implies $\varepsilon u\equiv1\pmod m$, and hence $u\equiv\pm1\pmod m$.
This contradicts the choice of $u$. Therefore $[M]\in \Bi_K(2N)\setminus P\Gamma_K(2N)$, and consequently $P\Gamma_K(2N) \subsetneq \Bi_K(2N)$.
Combining the two parts proves $P\Gamma_K(2N) \subsetneq \Bi_K(2N) \cong \operatorname{Aut}(X_{K,2N})$.
\end{proof}

\subsection{Mordell--Weil group}\label{subsec:MW}
We now examine the basic geometry of these $K3$ surfaces. These observations help us construct concrete models of $X_{K,2}$. From the structure of the lattice $S_K(2N)$, we can immediately deduce the following properties.

\begin{proposition}
    Let $X_{K,2N}$ be a $K3$ surface with $\NS(X_{K,2N}) \cong S_K(2N)$. Then:
    \begin{enumerate}
        \item Each of the primitive isotropic classes $e_1, e_2$ defines a genus-one fibration.
        \item $S_K(2N)$ does not represent $-2$; hence
        $X_{K,2N}$ contains no smooth rational curves.
    \end{enumerate}
\end{proposition}

To provide a useful contrast, we first recall Totaro's non-arithmeticity criterion for $K3$ surfaces:

\begin{theorem}\cite[Corollary~6.2]{Tot09}
\label{cor:Totaro_K3_non_arithmetic}
Let $X$ be a $K3$ surface with Picard number $\rho(X)\geq 4$. Suppose that (1) $X$ admits a genus-one fibration with no reducible fibers, (2) it admits a second genus-one fibration with positive Mordell--Weil rank, and (3) $X$ contains a $(-2)$-curve. Then the automorphism group $ \Aut(X)\subset O(\Pic(X)) $
is not commensurable with an arithmetic group.
\end{theorem}

Our $K3$ surfaces $X_{K,2N}$ contrast with Totaro's situation: they satisfy conditions (1) and (2) but fail condition (3). The absence of smooth rational curves implies that the Weyl group is trivial ($W=1$). 

In the next section, we will review the basics of the Mordell--Weil group of the Jacobian fibration.
\begin{proposition}
\label{prop:cusp_translations_XK2N}
Let $X=X_{K,2N}$ be a very general member of the family
$\NS(X)\cong S_{K}(2N).$ Let $F\in \NS(X)$ be the primitive nef isotropic class
corresponding to a cusp of the Bianchi model, and let
\[
\pi_F:X\to \mathbb P^1
\]
be the induced genus-one fibration. Let $J_F\to \mathbb P^1$ be its associated
Jacobian fibration. Then the unipotent cusp subgroup of $\Aut(X)$ stabilizing
$F$ is naturally identified with the Mordell--Weil group of $J_F$: $U_F\cong \MW(J_F).$
Under the Bianchi identification of the cusp stabilizer at $F$, one has
\[
U_F
\cong
\left\{
\left[
\begin{pmatrix}
1&\lambda\\
0&1
\end{pmatrix}
\right]
:
\lambda\in 2N\mathcal O_K
\right\}
\cong
2N\mathcal O_K
\cong
\mathbb Z^2.
\]
\end{proposition}
\begin{proof}
Since $F^2=0$ and $F\neq 0$, the Hodge index theorem implies that $F^\perp/\mathbb ZF$
is negative definite. Hence $O(F^\perp/\mathbb ZF)$ is finite. Therefore any
unipotent element in the stabilizer of $F$ acts trivially on
$F^\perp/\mathbb ZF$.

We use the following standard facts about Mordell--Weil groups of elliptic
surfaces. For a Jacobian fibration $\phi:J\to \mathbb P^1$
with zero section, every section acts on the generic fiber by translation, and,
since a $K3$ surface is minimal, the resulting birational self-map extends to an
automorphism of $J$; see \cite[Section~4, equation~(4.1)]{S24}.
Equivalently, there is a natural embedding
$\MW(\phi)\hookrightarrow \Aut(J).$
For a genus-one fibration $\pi_F:X\to \mathbb P^1$, its associated Jacobian
fibration $J_F\to \mathbb P^1$ acts fiberwise on $X$ through the usual torsor
action of the relative Jacobian. Thus the fiberwise translation subgroup of
$\pi_F$ is identified with $\MW(J_F)$.

We now explain why this group is torsion-free in our situation. Let $\phi:J_F\to \mathbb P^1$
be the associated Jacobian fibration, let $f$ be the fiber class, and let $z$ be
the zero section. The classes $f,z$ span the hyperbolic plane $U_\phi=\langle f,z\rangle\subset \NS(J_F).$
Let $\Sigma_\phi$ denote the root lattice generated by the components of
reducible fibers disjoint from the zero section. By
\cite[Proposition~4.1]{S24}, $\Sigma_\phi$ is precisely the lattice
generated by those fiber components, and by
\cite[Theorem~4.3]{S24} one has an isomorphism
\[
\MW(J_F)\cong \NS(J_F)/(U_\phi\oplus \Sigma_\phi).
\]
In our case $X$ has no $(-2)$-curves. Hence the genus-one fibration
$\pi_F$ has no reducible fibers; equivalently the associated Jacobian fibration
has no reducible-fiber root lattice, so $\Sigma_\phi=0$. Therefore
\[
\MW(J_F)\cong \NS(J_F)/U_\phi.
\]
Since $U_\phi\cong U$ is unimodular, it is a primitive direct summand of
$\NS(J_F)$. Hence the quotient $\NS(J_F)/U_\phi$ is torsion-free. Consequently,
$\MW(J_F)=\MW(J_F)_{\free}.$
Combining this with the lattice-theoretic description of the unipotent cusp
stabilizer gives
\[
U_F\cong \MW(J_F).
\]

Finally, in the Bianchi model, the cusp stabilizer corresponding to $F$ is the
level-$2N$ parabolic subgroup. Its unipotent radical is
\[
\lambda\longmapsto
\left[
\begin{pmatrix}
1&\lambda\\
0&1
\end{pmatrix}
\right],
\qquad \lambda\in 2N\mathcal O_K.
\]
Since $\mathcal O_K$ is a rank-two $\mathbb Z$-lattice, this group is
isomorphic to $\mathbb Z^2$.
\end{proof}
\section{Deformations of Kummer surfaces}

In this section, we construct the level-two $K3$ surfaces $X_{K,2}$ as very general deformations of Kummer surfaces. The main point is that the non-exceptional lattice $S_K(2)$ is primitive in the orthogonal complement of the Kummer lattice. This allows us to deform away the exceptional Kummer classes while keeping precisely the rank-four lattice $S_K(2)$ algebraic.

\subsection{Kummer surfaces}

We first recall the classical construction of Kummer surfaces. Let $A$ be an abelian surface and let $\iota:A\longrightarrow A, \ x\longmapsto -x$ be the natural involution.

\begin{construction}\label{const:kummer}
Let $G=\{1,\iota\}\cong \mathbb Z/2\mathbb Z$ act on $A$, and set $Y=A/G$. The fixed locus of $\iota$ is the group $A[2]$ of two-torsion points, which has cardinality $16$. Hence $Y$ has exactly $16$ ordinary double points, one for each point of $A[2]$. Away from these singularities the quotient is étale in codimension one, and the canonical divisor satisfies $K_Y\sim 0$.

Let $\pi:X\longrightarrow Y$ be the minimal resolution. Since the singularities are of type $A_1$, the resolution is crepant, and therefore $K_X=\pi^*K_Y\sim 0$. The exceptional locus consists of $16$ disjoint smooth rational curves $E_1,\ldots,E_{16}$ with $E_i^2=-2$. The resulting smooth surface $X$ is called the \emph{Kummer surface} associated to $A$, and is denoted $X=\operatorname{Km}(A)$. It is a $K3$ surface.
\end{construction}

\begin{definition}\label{def:kummer_lattice}
Let $X=\operatorname{Km}(A)$. Let $E=\langle E_1,\ldots,E_{16}\rangle\cong \langle -2\rangle^{\oplus 16}$ be the lattice generated by the sixteen exceptional curves. Its saturation in $H^2(X,\mathbb Z)$ is called the \emph{Kummer lattice} and will be denoted by $\mathcal K$. Thus $\mathcal K=(E\otimes \mathbb Q)\cap H^2(X,\mathbb Z)$.
\end{definition}

The Kummer lattice accounts for the exceptional part of the Picard group. The part coming from the abelian surface itself lies in the orthogonal complement of $\mathcal K$. The following standard fact is the basic reason for the appearance of a factor of $2$ in Kummer constructions.

\begin{lemma}\label{lem:kummer_orthogonal_complement}
Let $X=\operatorname{Km}(A)$ and let $\mathcal K\subset H^2(X,\mathbb Z)$ be the Kummer lattice. Then there is a Hodge isometry
$\mathcal K^\perp \cong H^2(A,\mathbb Z)(2).$
Under this identification,
$\operatorname{NS}(X)\cap \mathcal K^\perp \cong \operatorname{NS}(A)(2).$
 Then $\operatorname{NS}(A)(2)$ is primitive in $\mathcal K^\perp$.
\end{lemma}
We now prove the existence of the required Kummer deformation.

\begin{proposition}\label{prop:kummer_deformation}
Let $X$ be a complex $K3$ surface, and let $S\subset \operatorname{NS}(X)$ be a primitive hyperbolic sublattice. Then there exists a deformation $X'$ of $X$ such that
$\operatorname{Pic}(X')\cong S.$
Moreover, for a very general such deformation, every class in $\Lambda_{K3}\setminus S$ is not of type $(1,1)$.
\end{proposition}

\begin{proof}
Consider the period domain
$$\Omega_S = \left\{ [\omega]\in \mathbb P(S^\perp\otimes \mathbb C) \mid (\omega,\omega)=0,\ (\omega,\overline{\omega})>0 \right\}.$$
For each integral class $\delta\in \Lambda_{K3}\setminus S$, the condition that $\delta$ remains algebraic is the hyperplane condition
$(\omega,\delta)=0.$
Since there are only countably many such $\delta$, a very general point of $\Omega_S$ avoids all these hyperplanes. By the surjectivity of the period map, such a period is realized by a marked $K3$ surface. By the Lefschetz $(1,1)$ theorem, its Picard lattice is exactly $S$.
\end{proof}
\subsection{Primitive lattices of abelian surfaces and the equational model}

Let $K = \mathbb{Q}(\sqrt{-d})$ be an imaginary quadratic field with ring of integers $\mathcal{O}_K$. We fix an elliptic curve $E_K \coloneqq \mathbb{C}/\mathcal{O}_K$, and set $A_K \coloneqq E_K \times E_K$. 

\begin{definition}\label{def:graphclass}
We define the standard geometric divisor classes on $A_K$:
\begin{enumerate}
    \item \emph{Fiber divisors:} The horizontal and vertical fibers $F_1 = E_K \times \{0\}$ and $F_2 = \{0\} \times E_K$.
    \item \emph{Graph divisors:} For $\alpha \in \mathcal{O}_K$, the graph of multiplication by $\alpha$, denoted $\Gamma_\alpha = \{(x, \alpha x) \mid x \in E_K\} \subset A_K$.
\end{enumerate}
\end{definition}
Equivalently, in the Hermitian matrix model for $\NS(A_K)$, we use the classes
$e= \left(\begin{smallmatrix} 1&0\\0&0 \end{smallmatrix}\right)$,
$f= \left(\begin{smallmatrix} 0&0\\0&1 \end{smallmatrix}\right)$, and
$P(z)= \left(\begin{smallmatrix} 0&z\\ \overline z&0 \end{smallmatrix}\right)$.
The classes $e$ and $f$ span a copy of $U$. If $\{1,\tau\}$ is an integral basis of $\mathcal O_K$, then the remaining two
classes may be taken to be $P(1)$ and $P(\tau)$. Hence $\NS(A_K)\cong U\oplus M_K$,
where $M_K$ is the negative-definite CM trace lattice. We write
$S_K:=\NS(A_K)\cong U\oplus M_K$ and $S_K(2):=\NS(A_K)(2)\cong U(2)\oplus M_K(2)$.

Let $X_K^{\operatorname{Km}}=\operatorname{Km}(A_K)$.
By Lemma~\ref{lem:kummer_orthogonal_complement}, the non-exceptional part of the
Picard lattice of $X_K^{\operatorname{Km}}$ is
$\NS(X_K^{\operatorname{Km}})\cap \mathcal K^\perp \cong \NS(A_K)(2) \cong S_K(2)$.
In particular, $S_K(2)$ is primitive in $\NS(X_K^{\operatorname{Km}})\cap \mathcal K^\perp$.

We now explain how this deformation is seen in the double-cover equation. Choose
a Weierstrass equation $y^2=4x^3-g_2x-g_3$ for $E_K$. The quotient map
$E_K\longrightarrow E_K/\{\pm 1\}\cong \mathbb P^1$
is a double cover branched over the three finite roots of the cubic and over the
point at infinity. Thus the correct homogeneous branch section is the binary
quartic $$B_K(X,Y)=(4X^3-g_2XY^2-g_3Y^3)Y \in H^0(\mathbb P^1,\mathcal O_{\mathbb P^1}(4)).$$
The singular Kummer quotient of $A_K=E_K\times E_K$ is therefore the double
cover $$Y_K^{\operatorname{sing}} = \left\{ S^2=B_K(X,Y)B_K(U,V) \right\} \subset \operatorname{Spec}_{\mathbb{P}^1\times\mathbb{P}^1} \operatorname{Sym}(\mathcal{O}(-2,-2)).$$
Its branch divisor is the union of four vertical and four horizontal fibers.
Their $16$ intersection points give the $16$ ordinary double points of
$Y_K^{\operatorname{sing}}$, and resolving them gives $X_K^{\operatorname{Km}}$.

A local equational deformation of this double cover is obtained by replacing the
branch equation by a nearby $(4,4)$-form: $$S^2=B_K(X,Y)B_K(U,V)+tH_{4,4}(X,Y;U,V),$$
where $H_{4,4}\in H^0(\mathbb P^1\times\mathbb P^1,\mathcal O(4,4))$.
For a general choice of $H_{4,4}$, the branch curve is smooth and the double
cover is a smooth $K3$ surface. The two ruling classes on
$\mathbb P^1\times\mathbb P^1$ pull back to classes $e,f$ satisfying
$e^2=f^2=0$ and $e\cdot f=2$, so this equational family naturally preserves a copy of $U(2)$.

However, a general $(4,4)$-deformation does not preserve the graph classes.
The additional graph classes are precisely $P(1)$ and $P(\tau)$.
Geometrically, these come from the graphs of the endomorphisms
$x\mapsto x$ and $x\mapsto \tau x$ on $E_K$. After passing to the Kummer quotient and resolving its singularities, their classes lie
in $\NS(X_K^{\operatorname{Km}})\cap \mathcal K^\perp \cong S_K(2)$.

\begin{proposition}
\label{prop:main_existence}
There exists a nonempty Noether--Lefschetz locus in the local
\((4,4)\)-double-cover family such that a very general member \(X\) of this
locus satisfies
\[
    \operatorname{Pic}(X)\cong S_K(2)=U(2)\oplus M_K(2).
\]
\end{proposition}

\begin{proof}
Let $X_K^{\operatorname{Km}}=\operatorname{Km}(E_K\times E_K).$
By Lemma~\ref{lem:kummer_orthogonal_complement}, $S_K(2)$
is a primitive sublattice of \(K^\perp\), hence as a primitive sublattice of the
marked \(K3\) lattice. The two fiber classes on \(E_K\times E_K\) give the
\(U(2)\)-summand, while the graph classes associated with the endomorphisms
\(1\) and \(\tau\) give the two additional classes $P(1),\,
    P(\tau).$

Now consider the period domain
\[
    \Omega_{S_K(2)}
    =
    \left\{
        [\omega]\in \mathbb{P}\bigl(S_K(2)^\perp\otimes \mathbb{C}\bigr)
        \ \middle|\
        (\omega,\omega)=0,\;(\omega,\overline{\omega})>0
    \right\}.
\]
For every integral class $\delta\in \Lambda_{K3}\setminus S_K(2),$ the condition that \(\delta\) remains algebraic is the hyperplane condition
$(\omega,\delta)=0.$
Since there are only countably many such classes \(\delta\), a very general
period point of \(\Omega_{S_K(2)}\) avoids all these hyperplanes. By the
surjectivity of the period map, there exists a \(K3\) surface \(X\), arbitrarily
close to \(X_K^{\operatorname{Km}}\) in the \(S_K(2)\)-polarized deformation
space, such that $\operatorname{Pic}(X)\cong S_K(2).$

It remains to see that such a surface belongs to the \((4,4)\)-double-cover
family. Let \(F_1,F_2\in \operatorname{Pic}(X)\) be the two primitive isotropic
classes spanning the \(U(2)\)-summand. Thus
\[
    F_1^2=F_2^2=0,
    \qquad
    F_1\cdot F_2=2.
\]
Since \(S_K(2)\) contains no classes of square \(-2\), the nef cone of \(X\)
coincides with the positive cone. After choosing the positive cone component,
we may assume that \(F_1\) and \(F_2\) are nef.

By \cite[Chapter 2 Proposition~3.10]{Huy16} the complete linear systems \(|F_1|\) and \(|F_2|\) define
genus-one fibrations
\[
    \varphi_i:X\longrightarrow \mathbb{P}^1,
    \qquad
    i=1,2.
\]
Consider the product morphism
$\varphi=(\varphi_1,\varphi_2):
    X\longrightarrow \mathbb{P}^1\times \mathbb{P}^1.$
Since \(F_1\cdot F_2=2\), the morphism \(\varphi\) is generically finite of
degree \(2\). Moreover, $\Phi$ contracts no curve. Indeed, if an irreducible curve $C$
were contracted, then $C\cdot F_1=C\cdot F_2=0.$
Hence $[C]$ would lie in the negative-definite orthogonal complement of
$U(2)$ in $S_K(2)$. Since $S_K(2)$ contains no $(-2)$-classes, this is
impossible for an irreducible curve on a $K3$ surface. Thus $\Phi$ is finite. Therefore \(X\) is a double cover of
\(\mathbb{P}^1\times\mathbb{P}^1\). Since \(K_X\sim 0\), the canonical bundle
formula for double covers shows that the branch divisor has class
$-2K_{\mathbb{P}^1\times\mathbb{P}^1}
    =
    \mathcal{O}_{\mathbb{P}^1\times\mathbb{P}^1}(4,4).$
Thus \(X\) is realized as a smooth double cover of
\(\mathbb{P}^1\times\mathbb{P}^1\) branched along a smooth curve of bidegree
\((4,4)\). Since the two nef classes \(F_1,F_2\) deform the two ruling classes of the
Kummer double cover, the resulting branch divisor is a small deformation of
\(B_K(X,Y)B_K(U,V)=0\). Thus this construction lies in the local
\((4,4)\)-double-cover family considered above.

Under the \(U(2)\)-polarized period map for the \((4,4)\)-double-cover family,
the conditions that \(P(1)\) and \(P(\tau)\) remain algebraic are precisely the
two Hodge equations
\[
    (\omega,P(1))=0,
    \qquad
    (\omega,P(\tau))=0.
\]
Hence they cut out a local Noether--Lefschetz locus in the
\((4,4)\)-double-cover family. The preceding construction shows that this
locus is nonempty. For a very general point of this locus, the period avoids all
hyperplanes associated with classes outside \(S_K(2)\), and therefore
\[
    \operatorname{Pic}(X)\cong S_K(2)=U(2)\oplus M_K(2).
\]

\end{proof}

\begin{figure}[htbp]
\centering
\begin{tikzcd}[row sep=large, column sep=large]
    & {Y_{K,2}} \\
    {X_{K,2}} & {A_K/\langle \pm 1\rangle}
    \arrow[squiggly, from=1-2, to=2-1]
    \arrow["{\operatorname{Bl}_{p_1,\ldots,p_{16}}}", from=1-2, to=2-2]
    \arrow[squiggly, from=2-2, to=2-1]
\end{tikzcd}
\caption{
The Kummer surface \(Y_{K,2}\) resolves the singular quotient
\(A_K/\langle \pm 1\rangle\), and both lie on the boundary of the deformation
family whose very general member is \(X_{K,2}\).
}
\label{fig:kummer-deformation}
\end{figure}
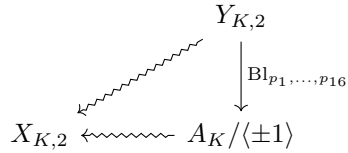

\subsection{Parabolic generation}

We now study the geometric generators of \(\operatorname{Aut}(X_{K,2})\) in the
cases $K=\mathbb{Q}(\sqrt{-2})$ \text{and} $\quad
    K=\mathbb{Q}(\sqrt{-7}).$
We begin by defining the relevant congruence subgroup and establishing how the projective group splits over the principal congruence subgroup.

\begin{definition}
\label{def:PSL-level-two}
Let \(K=\mathbb{Q}(\sqrt{-d})\). We define
\[
    \Gamma_K(2)
    :=
    \left\{
        [M]\in \operatorname{PSL}_2(\mathcal{O}_K)
        \ \middle|\
        \text{there exists } M\in \operatorname{SL}_2(\mathcal{O}_K)
        \text{ with } M\equiv I_2 \pmod{2\mathcal{O}_K}
    \right\}.
\]
\end{definition}

\begin{lemma}
\label{lem:PGamma-splits-over-Gamma}
Let \(K=\mathbb{Q}(\sqrt{-d})\) with \(d\in\{2,7\}\). Then $P\Gamma_K(2)
    =
    \Gamma_K(2)\rtimes \langle \eta\rangle,$
where \(\eta=[J]\) is represented by $J=
    \begin{pmatrix}
        1 & 0\\
        0 & -1
    \end{pmatrix}.$
Moreover, under the Hermitian action on \(S_K=\operatorname{Herm}_2(\mathcal{O}_K)\),
the element \(\eta\) acts by
\[
    e\longmapsto e,
    \qquad
    f\longmapsto f,
    \qquad
    P(z)\longmapsto -P(z).
\]
\end{lemma}

\begin{proof}
Since \(d\in\{2,7\}\), we have $\mathcal{O}_K^\times=\{\pm 1\}.$ Hence the determinant gives a well-defined homomorphism $\det:P\Gamma_K(2)\longrightarrow \{\pm 1\},$ because replacing a representative \(M\) by \(-M\) does not change
\(\det(M)\).

The kernel of this homomorphism is exactly \(\Gamma_K(2)\). Indeed, if
\([M]\in P\Gamma_K(2)\) and \(\det(M)=1\), then \(M\) may be chosen in
\(\operatorname{SL}_2(\mathcal{O}_K)\) with
\[
    M\equiv I_2 \pmod{2\mathcal{O}_K},
\]
so \([M]\in \Gamma_K(2)\). Conversely, every element of \(\Gamma_K(2)\) has
determinant \(1\).

The element $J=
    \begin{pmatrix}
        1 & 0\\
        0 & -1
    \end{pmatrix}$
satisfies $J\equiv I_2 \pmod{2\mathcal{O}_K}$ because \(-1\equiv 1\pmod{2\mathcal{O}_K}\), and $\det(J)=-1.$ Therefore \(\eta=[J]\) gives a splitting of the determinant map, and we obtain
\[
    P\Gamma_K(2)=\Gamma_K(2)\rtimes \langle \eta\rangle.
\]

Finally, for
\[
    e=
    \begin{pmatrix}
        1 & 0\\
        0 & 0
    \end{pmatrix},
    \qquad
    f=
    \begin{pmatrix}
        0 & 0\\
        0 & 1
    \end{pmatrix},
    \qquad
    P(z)=
    \begin{pmatrix}
        0 & z\\
        \overline{z} & 0
    \end{pmatrix},
\]
a direct computation gives
\[
    JeJ^*=e,
    \qquad
    JfJ^*=f,
    \qquad
    JP(z)J^*=-P(z).
\]
This proves the claim.
\end{proof}

One notable consequence of the decomposition $P\Gamma_K(2)
    =
    \Gamma_K(2)\rtimes\langle\eta\rangle$
is that the special-linear congruence subgroup contributes no torsion.
Every torsion element of \(P\Gamma_K(2)\) lies in the nontrivial coset
\(\Gamma_K(2)\eta\). Moreover, composing \(\eta\) with the parabolic
translations in \(\Gamma_K(2)\) produces infinitely many distinct
involutions. 
\begin{lemma}
\label{lem:level-two-torsion}
Let $K$ be an imaginary quadratic field.

\begin{enumerate}
    \item The group $\Gamma_K(2)$ is torsion-free.
    \item The group $P\Gamma_K(2)$ contains infinitely many distinct involutions.
    \item Every nontrivial finite-order element of $\operatorname{Bi}_K(2)$, and hence of $P\Gamma_K(2)$, has order $2$.
\end{enumerate}
\end{lemma}

\begin{proof}
We first prove that $\Gamma_K(2)$ is torsion-free. Let $[M] \in \Gamma_K(2)$ have finite order. By definition, we may choose a representative $M \in \operatorname{SL}_2(\mathcal{O}_K)$ satisfying $M \equiv I_2 \pmod{2\mathcal{O}_K}$.
Write $M = I_2 + 2A$ with $A \in M_2(\mathcal{O}_K)$. Since $\det(M) = 1$, we have
\[
    1 = \det(I_2 + 2A) = 1 + 2\operatorname{tr}(A) + 4\det(A),
\]
and therefore $\operatorname{tr}(M) = 2 + 2\operatorname{tr}(A) = 2 - 4\det(A)$.

Because $[M]$ has finite order in $\operatorname{PSL}_2(\mathcal{O}_K)$, some power of $M$ is equal to $\pm I_2$. Hence the eigenvalues of $M$ are roots of unity. Since $\det(M) = 1$, its trace is real, and thus $\operatorname{tr}(M) \in \mathcal{O}_K \cap \mathbb{R} = \mathbb{Z}$. Moreover, its eigenvalues have absolute value one, so $\operatorname{tr}(M) \in \{-2, -1, 0, 1, 2\}$. On the other hand,
\[
    \det(A) = \frac{2 - \operatorname{tr}(M)}{4} \in \mathcal{O}_K \cap \mathbb{Q} = \mathbb{Z}.
\]
Consequently, $\operatorname{tr}(M) \equiv 2 \pmod{4}$. Among the five possible traces above, this leaves only $\operatorname{tr}(M) = \pm 2$. A finite-order matrix over a field of characteristic zero is semisimple. Hence $M = I_2$ when $\operatorname{tr}(M) = 2$, and $M = -I_2$ when $\operatorname{tr}(M) = -2$. In either case, $[M]$ is the identity in $\operatorname{PSL}_2(\mathcal{O}_K)$. This proves~(1).

For $z \in \mathcal{O}_K$, set
\[
    U_z := \begin{pmatrix} 1 & 2z \\ 0 & 1 \end{pmatrix} \in \Gamma_K(2), \qquad J := \begin{pmatrix} 1 & 0 \\ 0 & -1 \end{pmatrix}.
\]
Then
\[
    U_z J = \begin{pmatrix} 1 & -2z \\ 0 & -1 \end{pmatrix} \equiv I_2 \pmod{2\mathcal{O}_K},
\]
and therefore $\tau_z := [U_z J] \in P\Gamma_K(2)$. Furthermore, $(U_z J)^2 = I_2$, so $\tau_z$ is an involution.

The elements $\tau_z$ are pairwise distinct. Indeed, if $\tau_z = \tau_w$ in $\operatorname{PGL}_2(\mathcal{O}_K)$, then
\[
    \begin{pmatrix} 1 & -2z \\ 0 & -1 \end{pmatrix} = u \begin{pmatrix} 1 & -2w \\ 0 & -1 \end{pmatrix}
\]
for some $u \in \mathcal{O}_K^\times$. Comparing the $(1,1)$-entries gives $u = 1$, and then comparing the $(1,2)$-entries gives $z = w$. Since $\mathcal{O}_K$ is infinite, this proves~(2).

We finally determine the possible orders of torsion elements in $\operatorname{Bi}_K(2)$. Let $[M] \in \operatorname{Bi}_K(2)$ have finite order, and choose a representative $M \in \operatorname{GL}_2(\mathcal{O}_K)$. Since $[M]$ lies in the projective level-two kernel, there exists $a \in \mathcal{O}_K$ such that $M \equiv a I_2 \pmod{2\mathcal{O}_K}$.
Set $t := \operatorname{tr}(M)$ and $\delta := \det(M) \in \mathcal{O}_K^\times$. The preceding congruence gives $t \equiv 2a \equiv 0 \pmod{2\mathcal{O}_K}$, so $t \in 2\mathcal{O}_K$.

Let $\lambda_1$ and $\lambda_2$ be the eigenvalues of $M$. Since $[M]$ has finite order in the projective group, the ratio $\zeta := \frac{\lambda_1}{\lambda_2}$ is a root of unity. Consider the projectively invariant quantity $q := \frac{t^2}{\delta}$. Using $t = \lambda_1 + \lambda_2$ and $\delta = \lambda_1 \lambda_2$, we obtain $q = 2 + \zeta + \zeta^{-1}$.
In particular, $q$ is real and $0 \leq q \leq 4$. Moreover, since $\delta$ is a unit, $q \in \mathcal{O}_K$. Thus $q \in \mathcal{O}_K \cap \mathbb{R} = \mathbb{Z}$. Since $t \in 2\mathcal{O}_K$, we also have $q = \frac{t^2}{\delta} \in 4\mathcal{O}_K$. Therefore $q \in 4\mathcal{O}_K \cap \mathbb{Z} = 4\mathbb{Z}$. Together with $0 \leq q \leq 4$, this gives $q \in \{0, 4\}$.

If $q = 4$, then $\zeta + \zeta^{-1} = 2$, and hence $\zeta = 1$. Since a projective finite-order matrix is semisimple, $M$ is scalar, so $[M]$ is the identity.

Thus every nontrivial finite-order element satisfies $q = 0$. It follows that $t = 0$, and the Cayley--Hamilton identity gives $M^2 - tM + \delta I_2 = 0$, hence $M^2 = -\delta I_2$. Therefore $[M]^2 = 1$ in $\operatorname{PGL}_2(\mathcal{O}_K)$. This proves~(3).
\end{proof}

The next point is to identify the covering involution in the Hermitian
lattice model.

\begin{lemma}
\label{lem:covering-involution-Hermitian-action}
Let \(X\) be a very general member of the \(S_K(2)\)-polarized
Noether--Lefschetz locus in the \((4,4)\)-double-cover family $\pi:X\longrightarrow \mathbb{P}^1\times\mathbb{P}^1.$  Let $\iota:X\longrightarrow X$ be the covering involution. Under the identification
\[
    \operatorname{Pic}(X)\cong S_K(2)
    =
    U(2)\oplus M_K(2),
\]
where \(U(2)\) is generated by the pullbacks of the two ruling classes and
\(M_K(2)\) is generated by the graph classes \(P(1)\) and \(P(\tau)\), the
action of \(\iota^*\) is given by
\[
    e\longmapsto e,
    \qquad
    f\longmapsto f,
    \qquad
    P(z)\longmapsto -P(z).
\]
Equivalently, \(\iota^*\) agrees with the Hermitian action of $J=
    \begin{pmatrix}
        1 & 0\\
        0 & -1
    \end{pmatrix}.$
\end{lemma}

\begin{proof}
The quotient of \(X\) by the involution \(\iota\) is
\(\mathbb{P}^1\times\mathbb{P}^1\). Hence the invariant part of
\(H^2(X,\mathbb{Q})\) under \(\iota^*\) is
\[
    H^2(X,\mathbb{Q})^{\iota^*}
    =
    \pi^*H^2(\mathbb{P}^1\times\mathbb{P}^1,\mathbb{Q}).
\]
In particular, on \(\operatorname{Pic}(X)\), the invariant sublattice is
generated by the pullbacks of the two ruling classes. These are precisely the
classes \(e\) and \(f\), and they span the \(U(2)\)-summand.

Since \(X\) is very general in the \(S_K(2)\)-polarized Noether--Lefschetz
locus, we have $\operatorname{Pic}(X)=U(2)\oplus M_K(2).$
Therefore the orthogonal complement of \(U(2)\) in \(\operatorname{Pic}(X)\) is
exactly \(M_K(2)\), generated by \(P(1)\) and \(P(\tau)\). Because the
\(\iota^*\)-invariant part of \(\operatorname{Pic}(X)_{\mathbb{Q}}\) is already
\(U(2)_{\mathbb{Q}}\), the involution \(\iota^*\) acts as \(-\operatorname{id}\)
on \(M_K(2)_{\mathbb{Q}}\). Since \(M_K(2)\) is an integral lattice preserved by
\(\iota^*\), this gives
\[
    \iota^*P(1)=-P(1),
    \qquad
    \iota^*P(\tau)=-P(\tau).
\]

On the other hand, the Hermitian action of $J=
    \begin{pmatrix}
        1 & 0\\
        0 & -1
    \end{pmatrix}$ sends
\[
    e\longmapsto e,
    \qquad
    f\longmapsto f,
    \qquad
    P(z)\longmapsto -P(z).
\]
Thus \(\iota^*\) agrees with the action of \(J\) on \(S_K(2)\).
\end{proof}

We now establish that the congruence subgroup is generated by parabolic elements, count the relevant cusps, and map these parabolic stabilizers to the Mordell-Weil groups of the genus-one fibrations.

\begin{theorem}
\label{thm:psl-level-two-parabolic-generation}
Let \(K=\mathbb{Q}(\sqrt{-d})\) be an imaginary quadratic field of class
number one, and assume \(d\notin\{1,3\}\). Then the level-two principal
congruence subgroup \(\Gamma_K(2)\) is generated by its parabolic elements if
and only if $d\in\{2,7\}.$
\end{theorem}

\begin{proof}
By the theorem of Baker, Goerner and Reid~\cite{BGR19}, among the
class-number-one imaginary quadratic fields, the level-two principal
congruence subgroup \(\Gamma_K(2)\) is generated by parabolic elements
precisely for $d\in\{1,2,3,7\}.$
Under the standing assumption \(d\notin\{1,3\}\), this list reduces exactly to $d\in\{2,7\}.$
This proves the claim.
\end{proof}

\begin{proposition} \label{prop:cusp_counting}
Let $K=\mathbb{Q}(\sqrt{-d})$ have class number one. Then the cusps of $\mathbb{H}^3/\Gamma_K(2)$ are naturally in bijection with $\mathbb{P}^1(R_2)$, where $R_2=\mathcal{O}_K/2\mathcal{O}_K.$ In particular, $\#\operatorname{Cusps} ( \mathbb{H}^3/\Gamma_{\mathbb{Q}(\sqrt{-2})}(2) ) = 6$, whereas $\#\operatorname{Cusps} ( \mathbb{H}^3/\Gamma_{\mathbb{Q}(\sqrt{-7})}(2) ) = 9$.
\end{proposition}

\begin{proof}
Since $K$ has class number one, the group $\operatorname{PSL}_2(\mathcal{O}_K)$ acts transitively on $\mathbb{P}^1(K)=K\cup\{\infty\}$. Passing to the principal congruence subgroup of level $2$, the cusp orbits are obtained by reduction modulo $2$. Therefore $\Gamma_K(2)\backslash \mathbb{P}^1(K) \cong \mathbb{P}^1(R_2).$

We now count this finite projective line in the two cases. First let $K=\mathbb{Q}(\sqrt{-2})$. Then $R_2 = \mathcal{O}_K/2\mathcal{O}_K \cong \mathbb{F}_2[t]/(t^2)$, where $t$ is the image of $\sqrt{-2}$. This is a local ring with maximal ideal $(t)$. Every point of $\mathbb{P}^1(R_2)$ is represented uniquely either in the form $[1:a]$ for $a\in R_2$, or in the form $[b:1]$ for $b\in (t)$. The first type gives $\#R_2=4$ points, and the second type gives $\#(t)=2$ points. Hence $\#\mathbb{P}^1(R_2)=4+2=6.$

Next let $K=\mathbb{Q}(\sqrt{-7})$. Since $2$ splits in $\mathcal{O}_K$, we have $R_2 = \mathcal{O}_K/2\mathcal{O}_K \cong \mathbb{F}_2\times \mathbb{F}_2.$ Therefore $\mathbb{P}^1(R_2) \cong \mathbb{P}^1(\mathbb{F}_2)\times \mathbb{P}^1(\mathbb{F}_2).$ Since $\#\mathbb{P}^1(\mathbb{F}_2)=3$, we obtain $\#\mathbb{P}^1(R_2)=3\cdot 3=9.$ This proves the claim.
\end{proof}

\begin{lemma} \label{lem:parabolic_mw}
Let $X$ be a very general $S_K(2)$-polarized K3 surface, and let $\ell_c\in S_K(2)$ be a primitive nef isotropic class. Let $\varphi_{\ell_c}:X\longrightarrow \mathbb{P}^1$ be the genus-one fibration defined by $\ell_c$. Then the translation subgroup of the parabolic stabilizer of $\ell_c$ is naturally identified with the Mordell--Weil group $\operatorname{MW}(J_{\ell_c}).$
\end{lemma}

\begin{proof}
The primitive nef isotropic class $\ell_c$ defines a genus-one fibration $\varphi_{\ell_c}:X\longrightarrow \mathbb{P}^1.$ The stabilizer of the ray $\mathbb{R}_{\geq 0}\ell_c$ in the automorphism group is the automorphism group of this genus-one fibration. Its parabolic translation subgroup consists exactly of automorphisms acting fiberwise by translations. For a genus-one fibration on a $K3$ surface, the group of fiberwise translations is the Mordell--Weil group of the fibration. Hence this parabolic translation subgroup is identified with $\operatorname{MW}(J_{\ell_c})$.
\end{proof}

Finally, we synthesize the algebraic splitting, the geometric covering involution, and the parabolic generators into a complete description of the automorphism group.

\begin{corollary} \label{cor:final_generation}
Let $K=\mathbb{Q}(\sqrt{-d})$ with $d\in\{2,7\}$, and let $X_{K,2}$ be a very general $S_K(2)$-polarized double-cover deformation of $\operatorname{Km}(E_K\times E_K)$. Then $\operatorname{Aut}(X_{K,2}) \cong \Gamma_K(2)\rtimes \langle \iota\rangle$, where $\iota$ is the covering involution. Moreover, the subgroup $\Gamma_K(2)$ is generated by the Mordell--Weil translation groups associated to genus-one fibrations: $\Gamma_K(2) = \langle \operatorname{MW}(J_{\ell_c}) \mid c\in \Gamma_K(2)\backslash \mathbb{P}^1(K) \rangle.$

For $K=\mathbb{Q}(\sqrt{-2})$, there are $6$ cusp orbits. The standard cusp translation subgroup is generated by $\begin{pmatrix} 1&2\\ 0&1 \end{pmatrix}$ and $\begin{pmatrix} 1&2\sqrt{-2}\\ 0&1 \end{pmatrix}.$ For $K=\mathbb{Q}(\sqrt{-7})$, there are $9$ cusp orbits. The standard cusp translation subgroup is generated by $\begin{pmatrix} 1&2\\ 0&1 \end{pmatrix}$ and $\begin{pmatrix} 1&1+\sqrt{-7}\\ 0&1 \end{pmatrix}.$
\end{corollary}

\begin{proof}
By Lemma~\ref{lem:PGamma-splits-over-Gamma}, we have $P\Gamma_K(2) = \Gamma_K(2)\rtimes \langle [J]\rangle,$ and $[J]$ is realized geometrically by the covering involution $\iota$. Using the Torelli identification $\operatorname{Aut}(X_{K,2}) \cong P\Gamma_K(2),$ we therefore obtain $\operatorname{Aut}(X_{K,2}) \cong \Gamma_K(2)\rtimes \langle \iota\rangle.$

By Theorem~\ref{thm:psl-level-two-parabolic-generation}, the group $\Gamma_K(2)$ is generated by its parabolic elements for $d\in\{2,7\}$. Each cusp corresponds to a primitive isotropic ray in the positive cone, hence to a genus-one fibration on $X_{K,2}$. By Lemma~\ref{lem:parabolic_mw}, the parabolic translation subgroup at this cusp is the Mordell--Weil group of the corresponding genus-one fibration. Therefore $\Gamma_K(2) = \langle \operatorname{MW}(J_{\ell_c}) \mid c\in \Gamma_K(2)\backslash \mathbb{P}^1(K) \rangle.$

Finally, the number of cusp orbits is computed in Proposition~\ref{prop:cusp_counting}. It is $6$ for $K=\mathbb{Q}(\sqrt{-2})$ and $9$ for $K=\mathbb{Q}(\sqrt{-7})$. The displayed matrices are the standard translation generators at the cusp $\infty$, corresponding respectively to the rank-two lattices $2\mathcal{O}_{\mathbb{Q}(\sqrt{-2})}$ and $2\mathcal{O}_{\mathbb{Q}(\sqrt{-7})}$.
\end{proof}

\section{Complete-intersection models}
\label{subsec:rank-four-complete-intersections}
We now treat the two fields 
\(K=\mathbb{Q}(i)\) and \(K=\mathbb{Q}(\sqrt{-3})\). In these cases the surfaces
\(X_{K,2}\) admit especially symmetric complete-intersection models, and the
corresponding automorphism groups are generated by natural deck involutions.

\begin{definition}
\label{def:type-I-II}
Let $\mathbb P_I:=\mathbb P^5\times(\mathbb P^1)^3,$
and let $h,e_1,e_2,e_3$ denote the hyperplane classes pulled back from the four factors. A $K3$ surface of \emph{type I} is a smooth complete intersection
$$X_I=D_{11}\cap D_{12}\cap D_{21}\cap D_{22}\cap D_{31}\cap D_{32}\subset \mathbb P_I,$$
where
$$D_{i1},D_{i2}\in |h+e_i|, \qquad i=1,2,3.$$

Let $\mathbb P_{II}:=(\mathbb P^1)^4,$
and let $g_1,g_2,g_3,g_4$ denote the hyperplane classes pulled back from the four factors. A $K3$ surface of \emph{type II} is a smooth complete intersection
$$X_{II}=D_1\cap D_2\subset \mathbb P_{II}, \qquad D_1,D_2\in |g_1+g_2+g_3+g_4|.$$
\end{definition}

By the adjunction formula, these are indeed $K3$ surfaces.

\begin{proposition}
\label{prop:type-I-intersection-lattice}
Let $X_I$ be a type I surface, and set
$$H:=h|_{X_I}, \qquad E_i:=e_i|_{X_I} \quad (i=1,2,3).$$
The lattice $G_I = \langle H, E_1, E_2, E_3 \rangle$ generated by these restrictions has the Gram matrix
$$G_I=
\begin{pmatrix}
8&4&4&4\\
4&0&2&2\\
4&2&0&2\\
4&2&2&0
\end{pmatrix}.$$
Moreover, $G_I\cong S_{\mathbb Q(i)}(2) $, and $S_{\mathbb Q(i)}(2)$ is primitive in $\mathrm{Pic}(X_I)$.
\end{proposition}

\begin{proof}
In the Chow ring of $\mathbb P^5\times(\mathbb P^1)^3$, we have $h^6=0$ and $e_i^2=0$. The class of $X_I$ is $[X_I]=(h+e_1)^2(h+e_2)^2(h+e_3)^2$. Intersections on $X_I$ are computed by taking the coefficient of $h^5e_1e_2e_3$. This yields $H^2=8$, $H\cdot E_i=4$, $E_i^2=0$, and $E_i\cdot E_j=2$ (for $i\neq j$), resulting in the displayed Gram matrix. With respect to the unimodular change of basis $E_1$, $E_2$, $E_1+E_2-E_3$, $H-E_1-E_2-E_3$, the Gram matrix becomes $U(2)\perp \langle -4\rangle\perp \langle -4\rangle$.

To show primitivity, assume $S_{\mathbb Q(i)}(2)$ is not primitive. Then there exists an element $v \in S_{\mathbb Q(i)}^\vee(2) \setminus S_{\mathbb Q(i)}(2)$ such that $v \in \mathrm{Pic}(X_I)$ and $v^2 \in 2\mathbb{Z}$. A direct computation in the discriminant group \(A_{G_I}\) shows that the only
nonzero classes which can generate an even overlattice are, up to permutation of
\(E_1,E_2,E_3\), represented by $\frac{1}{2}E_1,\,
    \frac{1}{2}(H+E_1),\,
    \frac{1}{2}H.$ Each leads to a geometric contradiction:
\begin{enumerate}
    \item $v \equiv \frac{1}{2}E_1 \pmod{S_{\mathbb Q(i)}(2)}$. Set $w := \frac{1}{2}E_1 - E_2$. We have $w^2 = -2$, so by Riemann-Roch, either $w$ or $-w$ is effective. However, $w \cdot E_3 = -1$ and $(-w) \cdot E_2 = -1$, contradicting the nefness of the base-point free divisors $E_2$ and $E_3$.
    \item $v \equiv \frac{1}{2}(H + E_1) \pmod{S_{\mathbb Q(i)}(2)}$. Set $w := \frac{1}{2}E_1 - \frac{1}{2}H + E_2$. Here $w^2 = -2$, implying $w$ or $-w$ is effective. However, $w \cdot E_2 = -1$ and $(-w) \cdot E_3 = -1$, again contradicting nefness.
    \item $v \equiv \frac{1}{2}H \pmod{S_{\mathbb Q(i)}(2)}$. Set $w := \frac{1}{2}H - E_1$. We have $w^2 = -2$, so $w$ or $-w$ is effective. Note that $w \cdot H = (-w) \cdot H = 0$. Because the projection $X_I \to \mathbb P^5$ is an isomorphism onto a complete intersection of three quadrics, $H$ is strictly ample on $X_I$. Its intersection with any effective curve must be strictly positive, providing an immediate contradiction.
\end{enumerate}
Thus, no such fractional class exists, and $S_{\mathbb Q(i)}(2)$ is primitive.
\end{proof}

\begin{proposition}
\label{prop:type-II-intersection-lattice}
Let $X_{II}$ be a type II surface, and set
$$G_i:=g_i|_{X_{II}} \quad (i=1,2,3,4).$$
The lattice $G_{II} = \langle G_1, G_2, G_3, G_4 \rangle$ generated by these restrictions has the Gram matrix
$$G_{II}=
\begin{pmatrix}
0&2&2&2\\
2&0&2&2\\
2&2&0&2\\
2&2&2&0
\end{pmatrix}.$$
Moreover, $G_{II}\cong S_{\mathbb Q(\sqrt{-3})}(2)$ and $S_{\mathbb Q(\sqrt{-3})}(2)$ is primitive in $\mathrm{Pic}(X_{II})$.
\end{proposition}

\begin{proof}
In the Chow ring of $(\mathbb P^1)^4$, $g_i^2=0$ and the class of $X_{II}$ is $[X_{II}]=(g_1+g_2+g_3+g_4)^2$. Consequently, $G_i^2=0$. For $i\neq j$, $G_i\cdot G_j=[g_1g_2g_3g_4]\, g_ig_j(g_1+g_2+g_3+g_4)^2=2$. With respect to the unimodular change of basis $G_1$, $G_2$, $G_1+G_2-G_3$, $-G_1-G_2+G_4$, the Gram matrix becomes $U(2)\perp A_2(-2)$.

To show primitivity, assume $S_{\mathbb Q(\sqrt{-3})}(2)$ is not primitive. There must exist $v \in S_{\mathbb Q(\sqrt{-3})}^\vee(2) \setminus S_{\mathbb Q(\sqrt{-3})}(2)$ such that $v \in \mathrm{Pic}(X_{II})$ and $v^2 \in 2\mathbb{Z}$. A direct computation gives $v \equiv \frac{1}{2}(G_1 + G_2 + G_3 + G_4) \pmod{S_{\mathbb Q(\sqrt{-3})}(2)}$. Set
$$w := \frac{1}{2}G_1 + \frac{1}{2}G_2 - \frac{1}{2}G_3 - \frac{1}{2}G_4.$$
Then $w \in \mathrm{Pic}(X_{II})$ and $w^2 = -2$. By the Riemann-Roch theorem for $K3$ surfaces, either $w$ or $-w$ is an effective class. However, $w \cdot G_1 = -1$ and $(-w) \cdot G_3 = -1$. Because each $G_i$ is a pullback of a globally generated line bundle from the factors, they are nef and must intersect any effective class non-negatively. This geometric contradiction proves $S_{\mathbb Q(\sqrt{-3})}(2)$ is primitive.
\end{proof}

Next, we show that for very general members of these families, the Picard group is generated by the pullbacks of the hyperplane sections.
\begin{theorem}
\label{lem:period-maps-dominant}
For both the Type I and Type II families, the global period map is generically
dominant onto the corresponding lattice-polarized \(K3\) moduli space. Consequently, a
very general Type I surface has Picard lattice \( S_{\mathbb Q(i)}(2) \),
and a very general Type II surface has Picard lattice \( S_{\mathbb Q(\sqrt{-3})}(2) \).
\end{theorem}

\begin{proof}
Let \(G_I\) and \(G_{II}\) be the primitive lattices generated by the ambient
divisor classes on the Type I and Type II complete intersections. By
Propositions~\ref{prop:type-I-intersection-lattice} and
\ref{prop:type-II-intersection-lattice}, we have
\[
    G_I \cong S_{\mathbb Q(i)}(2),
    \qquad
    G_{II} \cong S_{\mathbb Q(\sqrt{-3})}(2),
\]
and both lattices have rank \(4\). Hence the corresponding lattice-polarized
\(K3\) moduli spaces have dimension \( 20-4=16 \).

We next compute the dimensions of the two parameter spaces modulo ambient
automorphisms. For Type I, the equations are given by three pencils
\[
    \langle D_{i1},D_{i2}\rangle
    \subset
    H^0\!\left(\mathbb P^5\times(\mathbb P^1)^3,\mathcal O(h+e_i)\right),
    \qquad i=1,2,3.
\]
Each vector space has dimension \(12\), so the parameter space is \( \operatorname{Gr}(2,12)^3 \).
After quotienting by \( \operatorname{PGL}_6\times \operatorname{PGL}_2^3 \),
the dimension is
\[
    3\dim\operatorname{Gr}(2,12)-\dim\operatorname{PGL}_6
    -3\dim\operatorname{PGL}_2
    =
    60-35-9
    =
    16.
\]
For Type II, the equations form a pencil in \( H^0\!\left((\mathbb P^1)^4,\mathcal O(1,1,1,1)\right) \),
which has dimension \(16\). Hence the parameter space is \( \operatorname{Gr}(2,16) \),
and after quotienting by \(\operatorname{PGL}_2^4\) its dimension is
\[
    \dim\operatorname{Gr}(2,16)-4\dim\operatorname{PGL}_2
    =
    28-12
    =
    16.
\]

It remains to show that the period maps have finite generic fibers. We claim
that a general surface in either family reconstructs its ambient complete
intersection model from the polarized \(K3\) surface together with the ambient
divisor classes.

For Type II, let \( X_{II}\subset \mathbb P_{II}:=(\mathbb P^1)^4 \)
be a complete intersection of two divisors of multidegree \((1,1,1,1)\). The
Koszul resolution of \(\mathcal I_{X_{II}}\) is
\[
    0
    \longrightarrow
    \mathcal O(-2,-2,-2,-2)
    \longrightarrow
    \mathcal O(-1,-1,-1,-1)^{\oplus 2}
    \longrightarrow
    \mathcal I_{X_{II}}
    \longrightarrow
    0.
\]
Tensoring by any coordinate divisor \(g_j\) gives line bundles with an
\(\mathcal O_{\mathbb P^1}(-1)\)-factor. Therefore, by the Künneth formula,
 $H^k\!\left(\mathbb P_{II},\mathcal I_{X_{II}}(g_j)\right)=0
    \quad
    \text{for } k=0,1.$
Thus
\[
    H^0\!\left(\mathbb P_{II},\mathcal O_{\mathbb P_{II}}(g_j)\right)
    \xrightarrow{\sim}
    H^0\!\left(X_{II},\mathcal O_{X_{II}}(g_j)\right).
\]
Hence the four pencils \(|g_1|,\ldots,|g_4|\) on \(X_{II}\) recover the map \( X_{II}\longrightarrow (\mathbb P^1)^4 \).
The image is the original complete intersection.

For Type I, let \( X_I\subset \mathbb P_I:=\mathbb P^5\times(\mathbb P^1)^3 \)
be the complete intersection cut out by the rank-six bundle
\[
    \mathcal E
    =
    \bigoplus_{i=1}^3 \mathcal O(h+e_i)^{\oplus 2}.
\]
The Koszul resolution of \(\mathcal I_{X_I}\) has terms \( \bigwedge^a \mathcal E^\vee \).
After tensoring by \(h\) or by one of the \(e_i\), every summand has either an
\(\mathcal O_{\mathbb P^1}(-1)\)-factor or a \(\mathbb P^5\)-factor
\(\mathcal O_{\mathbb P^5}(-m)\) with \(1\leq m\leq 5\), and hence has no
cohomology by the Künneth formula. Therefore
\[
    H^k\!\left(\mathbb P_I,\mathcal I_{X_I}(D)\right)=0
    \qquad
    \text{for } k=0,1
\]
and for \( D\in\{h,e_1,e_2,e_3\} \).
Consequently,
\[
    H^0\!\left(\mathbb P_I,\mathcal O_{\mathbb P_I}(D)\right)
    \xrightarrow{\sim}
    H^0\!\left(X_I,\mathcal O_{X_I}(D)\right)
\]
for \(D=h,e_1,e_2,e_3\). Thus the complete linear systems \( |h|,\ |e_1|,\ |e_2|,\ |e_3| \)
on \(X_I\) reconstruct the ambient embedding \( X_I\longrightarrow \mathbb P^5\times(\mathbb P^1)^3 \).

Thus a very general fiber of the period map is finite: indeed, once the ordered
ambient divisor classes are fixed by the lattice polarization, the above
Koszul reconstruction determines the ambient projective model and the defining
equations up to finitely many choices; forgetting the ordering can only quotient
by a finite group of lattice symmetries.

The source and target of each period map have dimension \(16\). Since the
period maps have finite generic fibers, their images also have dimension
\(16\). Hence the period maps are generically dominant. In particular, the
differential of the period map has rank \(16\) at a general point.

Finally, the locus where the Picard lattice strictly contains the primitive
ambient lattice is a countable union of proper Noether--Lefschetz divisors.
Therefore a very general member has no algebraic classes beyond the ambient
lattice. Hence
\[
    \operatorname{Pic}(X_I)=G_I\cong S_{\mathbb Q(i)}(2),
    \qquad
    \operatorname{Pic}(X_{II})=G_{II}\cong S_{\mathbb Q(\sqrt{-3})}(2).
\]
\end{proof}

We now produce involutions from pairs of genus-one pencils.

\begin{lemma}
\label{lem:degree-two-product-fibrations}
Let $X$ be a smooth projective $K3$ surface, and let $F,G\in \operatorname{NS}(X)$ be primitive, nef, isotropic divisor classes such that $F^2=G^2=0$ and $F\cdot G=2$. Then:
\begin{enumerate}
  \item The complete linear systems $|F|$ and $|G|$ define genus-one fibrations $\varphi_F, \varphi_G \colon X \to \mathbb P^1$.
  \item The product morphism $\Phi = (\varphi_F, \varphi_G) \colon X \longrightarrow \mathbb P^1\times\mathbb P^1$ is generically finite of degree $2$.
  \item The nontrivial deck transformation of $\Phi$ defines an automorphism of $X$ of order two.
\end{enumerate}
\end{lemma}

\begin{proof}
By \cite[Proposition~3.10]{Huy16} a primitive nef divisor $F$ on a $K3$ surface with $F^2=0$ defines a base-point-free genus-one pencil. These define the fibrations $\varphi_F$ and $\varphi_G$. Let $H_1, H_2$ be the ruling classes on $\mathbb P^1\times\mathbb P^1$. Then $\Phi^*H_1=F$ and $\Phi^*H_2=G$, which yields $\Phi^*(H_1+H_2)=F+G$. Since $(F+G)^2 = F^2+2F\cdot G+G^2 = 4 > 0$, the divisor $F+G$ is big, meaning $\Phi$ is generically finite. By the projection formula, $(F+G)^2 = \deg(\Phi)(H_1+H_2)^2$, which implies $4 = 2\deg(\Phi)$, so $\deg(\Phi)=2$. 

The degree-two function field extension $\mathbb C(\mathbb P^1\times\mathbb P^1) \subset \mathbb C(X)$ induces a nontrivial birational deck involution on $X$. Because $X$ is a smooth $K3$ surface, every birational self-map is a regular automorphism.
\end{proof}

\begin{proposition}
\label{prop:natural-involutions-type-I-II}
A very general type I surface admits $12$ natural deck involutions arising from pairs of isotropic classes with intersection $2$. A very general type II surface admits $6$ such natural deck involutions.

Furthermore, each distinguished isotropic class corresponds to
an ideal vertex of the associated regular ideal octahedron, respectively
regular ideal tetrahedron.
\end{proposition}

\begin{proof}
For type I, Proposition~\ref{prop:type-I-intersection-lattice} gives six
distinguished primitive isotropic classes
$E_1,E_2,E_3,H-E_1,H-E_2,H-E_3.$
They lie on the boundary of the positive cone, since each has square zero.
Equivalently, after projectivizing the positive cone, their rays determine
ideal boundary points of the associated hyperbolic space
$\mathbb H(S_I)
    =
    \{x\in S_I\otimes \mathbb R \mid x^2>0\}/\mathbb R_{>0}.$ The intersection table from Proposition~\ref{prop:type-I-intersection-lattice}
shows that, among these six classes, the three pairs
$\{E_i,H-E_i\}, i=1,2,3,$
are the only non-adjacent pairs, while every other pair has intersection
number \(2\). Therefore the graph whose vertices are the six isotropic rays
and whose edges are the pairs with intersection \(2\) is the complete graph on
six vertices with three disjoint edges removed. This is precisely the
one-skeleton of an octahedron, with opposite vertices given by $E_i$ \text{and} $H-E_i.$
Thus the six natural isotropic rays are naturally identified with the six
ideal vertices of an ideal octahedron. Since all edge-pairs have the same
intersection number, this is the regular ideal octahedron in the hyperbolic
model determined by \(S_I\). Hence the \(12\) pairs with intersection \(2\)
are exactly the \(12\) edges of this ideal regular octahedron.

For type II, the four ambient classes
$G_1,G_2,G_3,G_4$ satisfy
$G_i^2=0,\qquad G_i\cdot G_j=2 \quad (i\neq j).$
Thus their projective rays are four ideal boundary points of
\(\mathbb H(S_{II})\). Since every pair of distinct vertices has the same
intersection number, the complete graph on these four isotropic rays is the
one-skeleton of a regular ideal tetrahedron. Consequently the classes
\(G_1,G_2,G_3,G_4\) are naturally identified with the four ideal vertices of
this tetrahedron, and the six pairs \(\{G_i,G_j\}\) with \(i\neq j\) are exactly
its six edges.

Finally, by Lemma~\ref{lem:degree-two-product-fibrations}, every adjacent
pair of primitive nef isotropic classes \(F,G\) with \(F\cdot G=2\) gives a
degree-two morphism
$X \longrightarrow \mathbb P^1\times \mathbb P^1$ and hence a deck involution. Therefore the type I surface has one natural
deck involution for each edge of the ideal octahedron, namely \(12\), and the
type II surface has one natural deck involution for each edge of the ideal
tetrahedron, namely \(6\).
\end{proof}

We now bridge the deck involution with a purely algebraic edge half-turn on the lattice.

\begin{lemma}
\label{lem:deck-involution-edge-isometry}
Let $X$ be a smooth K3 surface with Picard rank $\rho(X)=4$, and let $f \colon X \longrightarrow \mathbb{P}^1 \times \mathbb{P}^1$ be a finite degree-two cover with deck involution $\iota$. Let $\alpha=f^*H_1$ and $\beta=f^*H_2$ be the pullbacks of the two ruling classes of $\mathbb{P}^1 \times \mathbb{P}^1$. Then $\alpha$ and $\beta$ are primitive isotropic classes satisfying $(\alpha,\beta)=2$, and the induced action of $\iota^*$ on $\operatorname{NS}(X)$ is an involutive isometry given by the formula:
$$ \iota^*(x) = -x + (x,\beta)\alpha + (x,\alpha)\beta $$
\end{lemma}

\begin{proof}
First, we determine the $\iota$-invariant subspace of the N\'eron--Severi group. For a finite morphism $f$ of degree two, the pullback and pushforward operators satisfy $f_*f^* = 2\operatorname{id}$ on $H^2(\mathbb{P}^1\times\mathbb{P}^1,\mathbb{Q})$ and $f^*f_* = \operatorname{id} + \iota^*$ on $H^2(X,\mathbb{Q})$. If a class $x \in H^2(X,\mathbb{Q})$ is $\iota$-invariant ($x = \iota^*x$), then:
$$ x = \frac{1}{2}(x + \iota^*x) = \frac{1}{2}f^*f_*x $$
Thus, every $\iota$-invariant cohomology class is pulled back from $\mathbb{P}^1\times\mathbb{P}^1$. Since the reverse inclusion is immediate (as $\iota$ acts trivially on the quotient), and $H^2(\mathbb{P}^1\times \mathbb{P}^1,\mathbb{Q}) = \mathbb{Q} H_1 \oplus \mathbb{Q} H_2$ where both classes are algebraic, we obtain:
$$ \operatorname{NS}(X)_{\mathbb{Q}}^{\iota} = f^*\operatorname{NS}(\mathbb{P}^1\times \mathbb{P}^1)_{\mathbb{Q}} = \mathbb{Q}\alpha \oplus \mathbb{Q}\beta $$

Next, consider the purely algebraic map defined by $\phi(x) := -x + (x,\beta)\alpha + (x,\alpha)\beta$. Direct calculation shows that $\phi(\alpha) = \alpha$ and $\phi(\beta) = \beta$. For any $x$ in the orthogonal complement $(\mathbb{Q}\alpha \oplus \mathbb{Q}\beta)^\perp$, we have $(x,\alpha) = (x,\beta) = 0$, meaning $\phi(x) = -x$. Since this gives an orthogonal direct-sum decomposition and $\phi$ preserves the intersection form on each summand, $\phi$ is an involutive isometry.

Finally, we characterize the geometric involution $\iota^*$ on $\operatorname{NS}(X)_{\mathbb{Q}}$. Because $\rho(X)=4$, the orthogonal complement $K := (\mathbb{Q}\alpha \oplus \mathbb{Q}\beta)^\perp$ has dimension $2$. The involution $\iota^*$ preserves $K$. Since we established that the invariant subspace $\operatorname{NS}(X)_{\mathbb{Q}}^{\iota}$ is exactly $\mathbb{Q}\alpha \oplus \mathbb{Q}\beta$, $\iota^*$ possesses no non-zero invariant vectors in $K$, so that it must act as $-\mathrm{id}_K$. Therefore, $\iota^*$ acts as the identity on $\mathbb{Q}\alpha \oplus \mathbb{Q}\beta$ and as minus the identity on $K$, 
which agrees with the behavior of the algebraic map $\phi$. We conclude that $\iota^*(x) = -x + (x,\beta)\alpha + (x,\alpha)\beta$. Under the hyperbolic realization of the positive cone, this involutive isometry corresponds exactly to the hyperbolic half-turn about the geodesic joining the two ideal points $[\alpha]$ and $[\beta]$.
\end{proof}

We now show that, in each case, these involutions generate the automorphism group $\Bi_K(2)$. We do this by comparing arithmetic indices and hyperbolic covolumes. By Lemma~\ref{lem:deck-involution-edge-isometry}, each natural deck involution associated
with a pair of ambient isotropic classes is the corresponding edge half-turn
in the hyperbolic model of the positive cone.
\subsection{$K=\mathbb Q(i)$}
Let $O\subset \mathbb H^3$ be a regular ideal octahedron. For each face $F$ of $O$, let $s_F$ denote the hyperbolic reflection in the totally geodesic plane containing $F$. Define $W_O := \langle s_F \mid F \text{ is a face of } O\rangle$. The reflections in the faces of a regular ideal octahedron generate the
Coxeter reflection group of the regular ideal octahedral tessellation of
\(\mathbb{H}^3\).

\begin{lemma}
\label{lem:edge-half turn}
Let $F_1$ and $F_2$ be two faces of $O$ meeting along an edge $e$. Then $s_{F_1}s_{F_2}=\iota_e$, where $\iota_e$ is the edge half-turn about $e$.
\end{lemma}

\begin{proof}
A regular ideal octahedron has dihedral angles of $\pi/2$. In hyperbolic space, the product of two reflections in totally geodesic planes meeting at angle $\theta$ is the elliptic rotation of angle $2\theta$ about their intersection. With $\theta=\pi/2$, $s_{F_1}s_{F_2}$ is the rotation of angle $\pi$ about $e$, which is precisely $\iota_e$.
\end{proof}

The involutions $\iota_{\alpha,\beta}$ from the previous section are precisely the edge half-turns of $O$. This description is encoded by the sign homomorphism $\varepsilon_O:W_O\to \{\pm 1\}$ defined by $\varepsilon_O(s_F)=-1$. We define the even Coxeter subgroup $W_O^+ := \ker(\varepsilon_O)$. Thus $W_O^+$ consists of elements expressible as products of an even number of face reflections. 

\begin{lemma}
\label{lem:octahedron-generation}
The even Coxeter subgroup $W_O^+$ is generated by the edge half-turns of the regular ideal octahedron $O$.
\end{lemma}

\begin{proof}
By definition, $W_O^+$ is generated by products of two face reflections. Since the face-adjacency graph of the octahedron is connected, it suffices to take products of reflections in adjacent faces. By Lemma~\ref{lem:edge-half turn}, the product of reflections in two adjacent faces is precisely the half-turn about their common edge.
\end{proof}

\begin{proposition}
\label{prop:even-octahedral-domain}
Choose a face $F_0$ of $O$, and let $s_{F_0}$ be the reflection in the plane containing $F_0$. The set $D := O\cup s_{F_0}(O)$ is a fundamental domain for the action of $W_O^+$ on $\mathbb H^3$.
\end{proposition}

\begin{proof}
Since $W_O^+$ has index two in $W_O$ with coset decomposition $W_O = W_O^+ \sqcup W_O^+ s_{F_0}$, we can regroup the tiling as
$$ \mathbb H^3 = \bigcup_{g\in W_O^+} g \left(O \cup s_{F_0}(O)\right) = \bigcup_{g\in W_O^+} g(D). $$
To see that distinct translates have disjoint interiors, suppose $\operatorname{Int}(gD) \cap \operatorname{Int}(hD) \neq \varnothing$ for some $g,h \in W_O^+$. This implies $\operatorname{Int}(g\delta O) \cap \operatorname{Int}(h\delta' O) \neq \varnothing$ for some $\delta, \delta' \in \{1, s_{F_0}\}$. Since $O$ is a fundamental domain for $W_O$, we must have $g\delta = h\delta'$, or $h^{-1}g = \delta' \delta^{-1}$. The left-hand side lies in $W_O^+$, but $s_{F_0} \notin W_O^+$. Thus $\delta'\delta^{-1}$ cannot be $s_{F_0}$, forcing $\delta'\delta^{-1} = 1$ and therefore $g = h$.
\end{proof}

\begin{proposition}
\label{thm:gaussian-level-2-quotient}
Let $G_{\mathrm{deck}} \subseteq \Bi_{\mathbb Q(i)}(2)$ denote the subgroup generated by the natural deck involutions. Then $G_{\mathrm{deck}} = \Bi_{\mathbb Q(i)}(2)$.
\end{proposition}

\begin{proof}
By Lemma~\ref{lem:octahedron-generation}, we identify $G_{\mathrm{deck}}$ with $W_O^+$. We will prove equality by comparing the arithmetic index $[\Bi_{\mathbb Q(i)}:\Bi_{\mathbb Q(i)}(2)]$ with the geometric index $[\Bi_{\mathbb Q(i)}:G_{\mathrm{deck}}]$ computed from the octahedral covolume. 

Let $R_i:=\mathbb Z[i]/(2).$
Since $(2)=-(1+i)^2,$
we have $R_i\cong \mathbb F_2[\epsilon]/(\epsilon^2), \qquad \epsilon=1+i.$
The reduction map $\mathrm{GL}_2(\mathbb Z[i])\longrightarrow \mathrm{GL}_2(R_i)$
is surjective. Hence $\Bi_{\mathbb Q(i)}/\Bi_{\mathbb Q(i)}(2) \cong \mathrm{PGL}_2(R_i).$
The quotient map $R_i\to \mathbb F_2$ induces a homomorphism
$$ \mathrm{PGL}_2(R_i)\longrightarrow \mathrm{PGL}_2(\mathbb F_2) \cong \mathrm{GL}_2(\mathbb F_2) \cong S_3. $$
Its kernel consists of projective classes of matrices of the form $I+\epsilon A, \quad A\in M_2(\mathbb F_2),$
where two such matrices define the same projective class if they differ by a scalar matrix $1+\epsilon\lambda, \qquad \lambda\in \mathbb F_2.$
Thus the kernel is $M_2(\mathbb F_2)/\mathbb F_2 I,$
which has order $8$.
Since $|\mathrm{PGL}_2(\mathbb F_2)|=|S_3|=6,$
we obtain $|\mathrm{PGL}_2(R_i)|=8\cdot 6=48.$
Therefore $[\Bi_{\mathbb Q(i)}:\Bi_{\mathbb Q(i)}(2)]=48.$

We now compare this with the geometric index. By Proposition~\ref{prop:even-octahedral-domain}, the group $G_{\mathrm{deck}} = W_O^+$ has fundamental polyhedron
$ D=O\cup s_{F_0}(O), $
where $O$ is the regular ideal octahedron. Hence
$\Vol(G_{\mathrm{deck}}\backslash \mathbb H^3)=\Vol(D)=2\Vol(O).$
By \cite[Section~4.5]{MaclachlanReid}
$$ \Vol(\Bi_{\mathbb Q(i)}\backslash \mathbb H^3) = \frac{1}{24}\Vol(O). $$
Therefore
$$ [\Bi_{\mathbb Q(i)}:G_{\mathrm{deck}}] = \frac{\Vol(G_{\mathrm{deck}}\backslash\mathbb H^3)} {\Vol(\Bi_{\mathbb Q(i)}\backslash\mathbb H^3)} = \frac{2\Vol(O)}{\Vol(O)/24} = 48. $$
Thus both the arithmetic and geometric indices are $48$. Since we already have the inclusion $G_{\mathrm{deck}} \subseteq \Bi_{\mathbb Q(i)}(2),$
and both subgroups have index $48$ in $\Bi_{\mathbb Q(i)}$, it follows that $G_{\mathrm{deck}} = \Bi_{\mathbb Q(i)}(2).$
\end{proof}

\subsection{The case \(K=\mathbb Q(\sqrt{-3})\)}
Let $ \omega=\frac{-1+\sqrt{-3}}{2}, \quad \mathcal{O}_K=\mathbb{Z}[\omega]. $ We use the regular ideal cubic honeycomb in \(\mathbb{H}^3\). A regular ideal cube has dihedral angle \(\pi/3\), so six cubes meet around each edge; in Schläfli notation this is the honeycomb $ \{4,3,6\}.$ The reflections in the square faces of one regular ideal cube generate the corresponding Coxeter reflection group. We use Coxeter's standard description of this honeycomb~\cite{CoxeterHoneycombs}. In the Eisenstein normalization, take the regular ideal tetrahedron \[ T=[\infty,0,1,-\omega]. \] The regular ideal cube can be obtained from \(T\) by adjoining the four tetrahedra adjacent to \(T\) across its four faces. Namely, if \(s_F\) denotes reflection in the plane containing a face \(F\subset T\), then \[ D = T\cup \bigcup_{F\subset T}s_F(T) \] is a regular ideal cube, subdivided into five regular ideal tetrahedra.

\begin{figure}[htbp]
    \centering
    \includegraphics[width=0.55\textwidth]{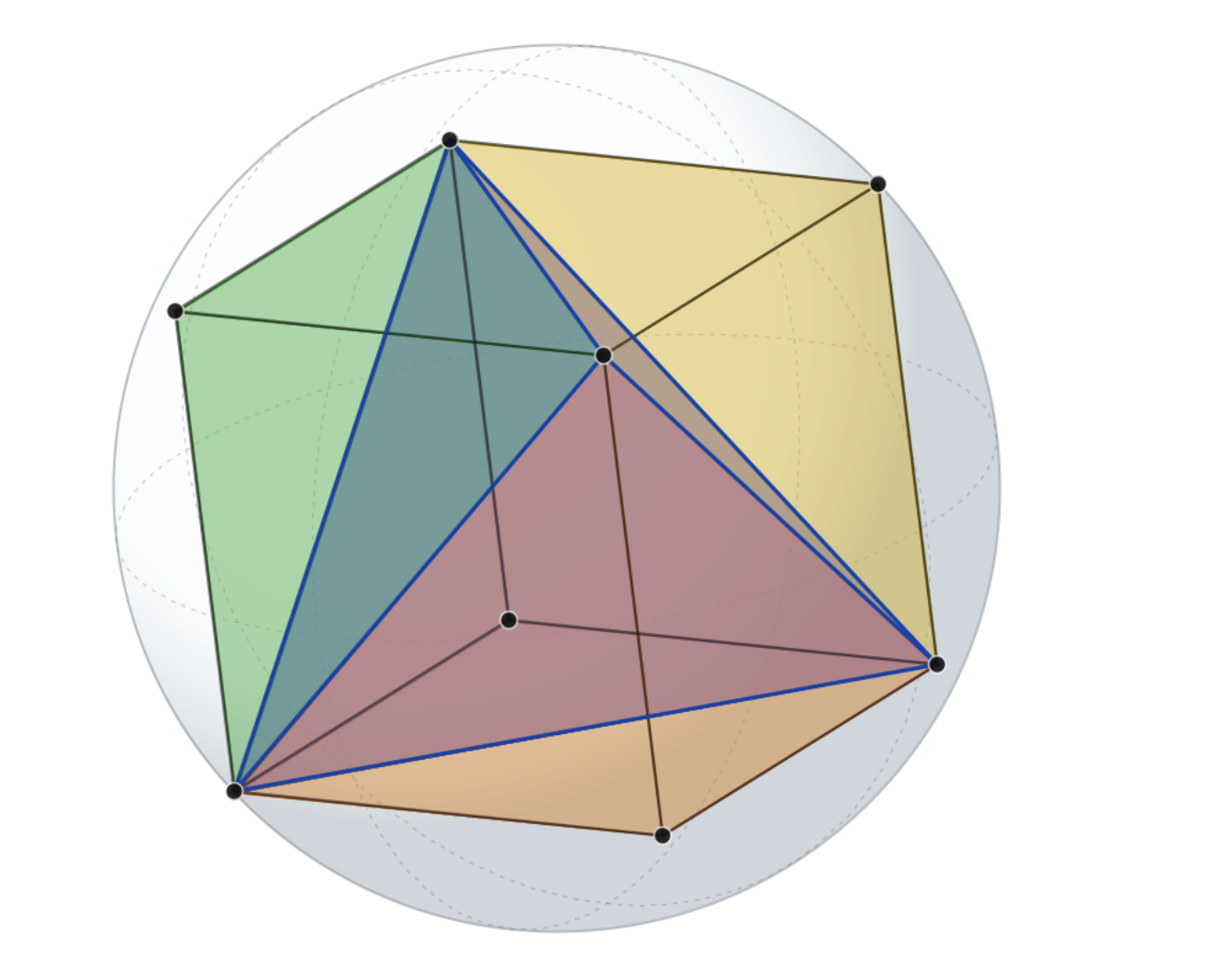}
    \caption{The union of the central ideal tetrahedron $T$ with its four adjacent Coxeter chambers $T_1, \dots, T_4$.}
\end{figure}

\begin{lemma}
\label{lem:tetrahedron-half-turns-domain}
Let \(T\subset \mathbb{H}^3\) be a regular ideal tetrahedron, and let
\[
    D:=T\cup \bigcup_{F\subset T}s_F(T),
\]
where \(s_F\) denotes reflection in the plane containing the face \(F\) of
\(T\). Then \(D\) is a regular ideal cube. Moreover, if \(e\) runs over the six
edges of \(T\), the half-turn \(\iota_e\) about \(e\) gives the side pairing of
the square face of \(D\) having \(e\) as a diagonal. Consequently, the group
\[
    G_T:=\langle \iota_e \mid e\subset T \text{ an edge}\rangle
\]
has \(D\) as a fundamental polyhedron.
\end{lemma}

\begin{proof}
By Coxeter's description of the regular ideal cubic honeycomb
\(\{4,3,6\}\), a regular ideal cube is subdivided into five regular ideal
tetrahedra: one central tetrahedron \(T\), together with the four tetrahedra
adjacent to \(T\) across its four faces~\cite{CoxeterHoneycombs}. Hence
\[
    D=T\cup \bigcup_{F\subset T}s_F(T)
\]
is a regular ideal cube.

The boundary of \(D\) consists of twelve triangular faces, grouped into six
squares. These six square faces are naturally indexed by the six edges of the
central tetrahedron \(T\). If \(e\subset T\) is an edge, let \(Q_e\) be the
corresponding square face of \(D\). Then \(e\) is a diagonal of \(Q_e\). The
half-turn \(\iota_e\) about \(e\) preserves the plane containing \(Q_e\) and
exchanges the two half-spaces bounded by this plane. Therefore \(\iota_e\)
preserves \(Q_e\) setwise and exchanges the two half-spaces bounded by
the plane containing \(Q_e\). Hence \(\iota_e(D)\) is the cube adjacent
to \(D\) across \(Q_e\).

Thus the six half-turns \(\iota_e\), one for each edge \(e\subset T\), are
exactly the side-pairing transformations of the six square faces of the cube
\(D\). Since the regular ideal cubic honeycomb \(\{4,3,6\}\) has dihedral angle
\(\pi/3\), six cubes meet around each edge. Hence the side pairings satisfy the
cycle conditions of the Poincar\'e polyhedron theorem~\cite[cf.~Theorem~H.11]{Maskit}. Therefore the group $G_T=\langle \iota_e \mid e\subset T\rangle$
is discrete and has \(D\) as a fundamental polyhedron.

\end{proof}

\begin{proposition}\label{prop:G-equals-Gamma2-volume}
Let $G_{\mathrm{deck}} \subseteq \Bi_{\mathbb Q(\sqrt{-3})}(2)$ denote the subgroup generated by the natural deck involutions. Then $G_{\mathrm{deck}} = \Bi_{\mathbb Q(\sqrt{-3})}(2)$.
\end{proposition}

\begin{proof}
By Corollary~\ref{cor:class-number-one-level-two}, $G_{\text{deck}}\subset \Bi_{\mathbb{Q}(\sqrt{-3})}$. 
By Lemma~\ref{lem:tetrahedron-half-turns-domain}, the deck involution group $G_{\mathrm{deck}}$ is precisely $G_T$. We verify the subgroup equality by confirming that $[\Bi_{\mathbb Q(\sqrt{-3})} : \Bi_{\mathbb Q(\sqrt{-3})}(2)]$ coincides with the geometric index $[\Bi_{\mathbb Q(\sqrt{-3})} : G_{\mathrm{deck}}]$.

First, we compute the arithmetic index of $\Bi_{\mathbb Q(\sqrt{-3})}(2)$. Since $2$ is inert in $\mathbb Z[\omega]$, we have $\mathbb Z[\omega]/(2)\cong \mathbb F_4$. The reduction map $\mathrm{PGL}_2(\mathbb Z[\omega])\longrightarrow \mathrm{PGL}_2(\mathbb F_4)$ is surjective, so
$ [\Bi_{\mathbb Q(\sqrt{-3})}:\Bi_{\mathbb Q(\sqrt{-3})}(2)] = |\mathrm{PGL}_2(\mathbb F_4)|. $
We calculate $|\operatorname{GL}_2(\mathbb F_4)| = (4^2-1)(4^2-4) = 15\cdot 12 = 180$. Projectivizing by the scalar subgroup of order $3$ gives $|\mathrm{PGL}_2(\mathbb F_4)|= 180/3 = 60$. Thus,
$ [\Bi_{\mathbb Q(\sqrt{-3})}:\Bi_{\mathbb Q(\sqrt{-3})}(2)]=60. $

On the other hand, we compute the geometric index. By Lemma~\ref{lem:tetrahedron-half-turns-domain}, $D$ is a fundamental domain for $G_{\mathrm{deck}}$. Since $D$ is the union of five regular ideal tetrahedra, $\Vol(G_{\mathrm{deck}}\backslash\mathbb H^3)=5v_3$. Using the covolume formula in $\mathrm{PGL}_2$ (see \cite[Chapter~1.4.3]{MaclachlanReid}), we have $\Vol(\Bi_{\mathbb Q(\sqrt{-3})}\backslash\mathbb H^3) = v_3/12$. We obtain the index:
$$ [\Bi_{\mathbb Q(\sqrt{-3})}:G_{\mathrm{deck}}] = \frac{5v_3}{v_3/12} = 60. $$
Since both indices equal $60$, the inclusion $G_{\mathrm{deck}} \subseteq \Bi_{\mathbb Q(\sqrt{-3})}(2)$ is an exact equality, yielding $G_{\mathrm{deck}} = \Bi_{\mathbb Q(\sqrt{-3})}(2)$.
\end{proof}
\begin{theorem}
\label{thm:deck-generates-automorphism}
For $K = \mathbb Q(i)$ and $K = \mathbb Q(\sqrt{-3})$, the group of natural deck involutions $G_{\mathrm{deck}}$ generates the full automorphism group $\operatorname{Aut}(X_{K,2})$.
\end{theorem}

\begin{proof}
By Proposition~\ref{thm:gaussian-level-2-quotient} and Proposition~\ref{prop:G-equals-Gamma2-volume}, the geometric subgroup generated by the deck involutions is precisely $G_{\mathrm{deck}} = \Bi_K(2)$ in both cases. By the arithmetic Torelli identification established in Section~3, we have $\operatorname{Aut}(X_{K,2}) \cong \Bi_K(2)$. Therefore, we conclude that $G_{\mathrm{deck}} = \operatorname{Aut}(X_{K,2})$.
\end{proof}
\FloatBarrier
\bibliographystyle{abbrv}
\bibliography{bib}
\end{document}